\documentclass[11pt,a4paper]{article}

\usepackage{tikz}
\usepackage{amsthm}
\usepackage{amsmath}
\usepackage{amssymb}
\usepackage{geometry}
\usepackage{graphicx}
\usepackage{amsfonts}
\usepackage{epstopdf}
\usepackage{enumerate}
\usepackage{hyperref}

\newcommand{\rref}{\mathrm{ref}} 

\newcommand{\vphi}{\varphi}
\newcommand{\veps}{\varepsilon}

\newcommand{\DD}{\mathbf{D}} 
\newcommand{\dd}{\partial}
\newcommand{\md}{\mathrm{d}}

\newcommand{\R}{\mathbb{R}}

\newcommand{\rmnum}[1]{\romannumeral #1} 
\newcommand{\Rmnum}[1]{\uppercase\expandafter{\romannumeral#1}} 

\newcommand{\Ee}{\mathcal{E}}
\newcommand{\Dd}{\mathcal{D}}
\newcommand{\Ff}{\mathcal{F}}

\newcommand{\PP}{\mathbb{P}} 
\newcommand{\EE}{\mathbb{E}} 
\newcommand{\indi}{\mathbf{1}} 

\newcommand{\myset}[1]{\left\{#1\right\}}
\newcommand{\mybar}[1]{\overline{#1}}

\newtheorem{mythm}{Theorem}[section]
\newtheorem{myprop}[mythm]{Proposition}
\newtheorem{mylem}[mythm]{Lemma}
\newtheorem{mycor}[mythm]{Corollary}
\newtheorem{myrmk}[mythm]{Remark}

\AtEndDocument{
\bigskip{\footnotesize%
  \textsc{Fakult\"{a}t f\"{u}r Mathematik, Universit\"{a}t Bielefeld, Postfach 100131, 33501 Bielefeld, Germany.} \par  
  \textit{E-mail address}: \texttt{grigor@math.uni-bielefeld.de} \par
}
\bigskip{\footnotesize%
  \textsc{Fakult\"{a}t f\"{u}r Mathematik, Universit\"{a}t Bielefeld, Postfach 100131, 33501 Bielefeld, Germany.} \par  
  \textit{E-mail address}: \texttt{ymeng@math.uni-bielefeld.de} \par
  and \par
  \addvspace{\medskipamount}
  \textsc{Department of Mathematical Sciences, Tsinghua University, Beijing 100084, China.} \par
  \textit{E-mail address}: \texttt{meng-yang13@mails.tsinghua.edu.cn}
}}

\begin{document}

\title{Determination of the Walk Dimension of the Sierpi\'nski Gasket Without Using Diffusion}
\author{Alexander Grigor'yan and Meng Yang}
\date{}

\maketitle

\abstract{We determine the walk dimension of the Sierpi\'nski gasket without using diffusion. We construct non-local regular Dirichlet forms on the Sierpi\'nski gasket from regular Dirichlet forms on the Sierpi\'nski graph whose suitable boundary is the Sierpi\'nski gasket.}

\footnote{\textsl{Date}: \today}
\footnote{The authors were supported by SFB 701 and SFB 1283 of the German Research Council (DFG).}

\section{Introduction}

It is well known that the Brownian motion in $\R^n$ is associated with the Dirichlet form
$$
\begin{cases}
\Ee(u,u)=\int_{\R^n}|\nabla u(x)|^2\md x,\\
\Dd[\Ee]=W^{1,2}(\R^n),
\end{cases}
$$
and the symmetric stable process in $\R^n$ of index $\beta$ is associated with the Dirichlet form
\begin{equation}\label{eqn_DFRn}
\begin{cases}
\Ee(u,u)=c_{n,\beta}\int_{\R^n}\int_{\R^n}\frac{(u(x)-u(y))^2}{|x-y|^{n+\beta}}\md x\md y,\\
\Dd[\Ee]=B^{\beta/2}_{2,2}(\R^n),
\end{cases}
\end{equation}
where $c_{n,\beta}>0$ is some normalizing constant. It is also known that $\beta$ can take arbitrary value in $(0,2)$. The symmetric stable process in $\R^n$ of index $\beta$ can be obtained from the Brownian motion in $\R^n$ by subordination technique, using the fact that the generator of the former is $(-\Delta)^{\beta/2}$ while the Laplace operator $-\Delta$ is the generator of the latter.

The main problem to be addressed in this paper is the range of the index of jump processes on more general spaces, notably, on fractals. We first define what we mean by index in a general setting.

Let $(M,d)$ be a locally compact separable metric space and $\mu$ be a Radon measure on $M$. Denote by $B(x,r)$ metric balls in $(M,d)$ and assume that $(M,d,\mu)$ is $\alpha$-regular in the sense that $\mu(B(x,r))\asymp r^\alpha$ for all $x\in M$ and $r\in(0,\mathrm{diam}(M))$. In particular, the Hausdorff dimension of $M$ is equal to $\alpha$ and the measure $\mu$ is equivalent to the Hausdorff measure of dimension $\alpha$ (see \cite{GHL03}).

Inspired by (\ref{eqn_DFRn}), consider the following quadratic form
\begin{equation}\label{eqn_E_K}
\begin{cases}
\Ee(u,u)=\int_M\int_M\frac{(u(x)-u(y))^2}{d(x,y)^{\alpha+\beta}}\mu(\md x)\mu(\md y),\\
\Ff=\myset{u\in L^2(M;\mu):\Ee(u,u)<+\infty},
\end{cases}
\end{equation}
where $\beta>0$ is so far arbitrary. By a general theory of Dirichlet form from \cite{FOT11}, in order for $(\Ee,\Ff)$ to be related to a jump process on $M$, $(\Ee,\Ff)$ has to be a \emph{regular Dirichlet form} on $L^2(M;\mu)$. In particular, $\Ff$ has to be dense in $L^2(M;\mu)$. In fact, it can happen that $\Ff=\myset{0}$ or $\Ff$ consists of constant functions (for example, if $M=\R^n$ and $\beta\ge2$, then $\Ff=\myset{0}$).

In all known examples, the range of $\beta$ for which $(\Ee,\Ff)$ is a regular Dirichlet form on $L^2(M;\mu)$ is an interval $(0,d_w)$ for some $d_w\in[0,+\infty]$. We refer to this value of $d_w$ as the \emph{walk dimension} of metric measure space $(M,d,\mu)$.

In fact, the walk dimension is an invariant of the metric space $(M,d)$. For example, the walk dimension of $\R^n$ is equal to 2 for all $n$. On most fractal spaces the walk dimension is strictly larger than 2. For example, on SG we have $d_w=\log5/\log2$.

To determine the walk dimension $d_w$, a common method is to use the diffusion on $M$ and its sub-Gaussian heat kernel estimate. Indeed, assume that a diffusion (corresponding to a local Dirichlet form) is constructed on $M$ and its heat kernel $p_t(x,y)$ (equivalently, the transition density) satisfies the following sub-Gaussian estimate
$$p_t(x,y)\asymp\frac{C}{t^{\frac{\alpha}{\gamma}}}\exp{\left(-c\left(\frac{d(x,y)^\gamma}{t}\right)^{\frac{1}{\gamma-1}}\right)}$$
at least for a bounded range of time $t$ and for all $x,y\in M$. Such estimates are known for many fractals, see for example, \cite{BB92,BP88,Kum93}. Here the parameter $\alpha$ is the Hausdorff dimension of $(M,d)$ as above and the parameter $\gamma$ is called the \emph{walk dimension} of the diffusion. For example, for Sierpi\'nski gasket (SG) we have $\gamma=\log5/\log2$ and for Sierpi\'nski carpet $\gamma\approx2.097$ (the exact value of $\gamma$ in this case is not known). Denote by $\mathcal{L}$ the positive definite generator of this diffusion. Then, for all $\delta\in(0,1)$, the operator $\mathcal{L}^\delta$ generates a jump process with a jump kernel
$$J(x,y)\asymp d(x,y)^{-(\alpha+\beta)},$$
where $\beta=\delta\gamma$ (see \cite{Kum03}). Hence, $\beta$ can take all values in $(0,\gamma)$. Using sub-Gaussian heat kernel estimate, one shows that for $\beta>\gamma$, $\Ff$ consists of constant functions see \cite{Pie00}. Hence the walk dimension $\gamma$ of the diffusion coincides with the walk dimension $d_w$ of the metric measure space. In particular, for SG we have $d_w={\log5}/{\log2}$.

The purpose of this paper is to provide an alternative method to determine this value of the walk dimension of SG \emph{without} using diffusion. We hope that this method will apply also to more general settings thus providing a direct way of determination of the range of the index.

Let $M=K$ be Sierpi\'nski gasket (SG) in $\R^2$ endowed with metric $d(x,y)=|x-y|$ and measure $\mu=\nu$ normalized Hausdorff measure on $K$. Let $(\Ee_K,\Ff_K)$ be given by Equation (\ref{eqn_E_K}) where $\alpha=\log3/\log2$ is Hausdorff dimension of SG and $\beta>0$ is some parameter.

Our approach is based on a recent paper \cite{KLW17} of S.-L. Kong, K.-S. Lau and T.-K. Wong. They introduced conductances with parameter $\lambda\in(0,1)$ on the Sierpi\'nski graph $X$ to obtain a random walk (and a corresponding energy form) on $X$ and showed that the Martin boundary of that random walk is homeomorphic to $K$. Let $\mybar{X}$ be the Martin compactification of $X$. It was also proved in \cite{KLW17} that the energy form on $X$ induces an energy form on $K\cong\mybar{X}\backslash X$ of the form (\ref{eqn_E_K}) with $\beta=-\log\lambda/\log2$. However, no restriction on $\beta$ was established, so that above energy form on $K$ does not have to be a regular Dirichlet form.

In this paper, we establish the exact restriction on $\lambda$ (hence on $\beta$) under which $(\Ee_K,\Ff_K)$ is a regular Dirichlet form on $L^2(K;\nu)$. Our method is as follows.

Firstly, we introduce a measure $m$ on $X$ to obtain a regular Dirichlet form $(\Ee_X,\Ff_X)$ on $L^2(X;m)$ associated with above random walk on $X$. Then we extend this Dirichlet form to an \emph{active reflected} Dirichlet form $(\Ee^\rref,\Ff^\rref_a)$ on $L^2(X;m)$ which is not regular, though.

Secondly, we \emph{regularize} $(\Ee^\rref,\Ff^\rref_a)$ on $L^2(X;m)$ using the theory of \cite{Fuk71}. The result of regularization is a regular Dirichlet form $(\Ee_{\mybar{X}},\Ff_{\mybar{X}})$ on $L^2(\mybar{X};m)$ that is an extension of $(\Ee_{{X}},\Ff_{{X}})$ on $L^2({X};m)$. By \cite{Fuk71}, regularization is always possible, but we show that the regularized form ``sits" on $\mybar{X}$ provided $\lambda>1/5$ which is equivalent $\beta<\beta^*:=\log5/\log2$.

Thirdly, we take trace of $\Ee_{\mybar{X}}$ to $K$ and obtain a regular Dirichlet form $(\Ee_K,\Ff_K)$ on $L^2(K;\nu)$ of the form (\ref{eqn_E_K}).

If $\beta>\beta^*$, then we show directly that $\Ff_K$ consists only of constant functions. Hence we conclude that $d_w=\beta^*=\log5/\log2$. This approach allows to detect the critical value $d_w$ of the index $\beta$ of the jump process without construction of diffusion.

So far this approach has been realized on SG but we plan to extend this method to a large family of fractals.

This paper is organized as follows. In section \ref{sec_SG}, we review basic constructions of Sierpi\'nski gasket $K$ and Sierpi\'nski graph $X$. In section \ref{sec_rw}, we give a transient reversible random walk $Z$ on $X$. In section \ref{sec_df_X}, we construct a regular Dirichlet form $\Ee_X$ on $X$ and its corresponding symmetric Hunt process $\myset{X_t}$. We prove that the Martin boundaries of $\myset{X_t}$ and $Z$ coincide. We show that $\Ee_X$ is stochastically incomplete and $\myset{X_t}$ goes to infinity in finite time almost surely. In section \ref{sec_ref}, we construct active reflected Dirichlet form $(\Ee^{\rref},\Ff^{\rref}_a)$ and show that $\Ff_X\subsetneqq\Ff^{\rref}_a$, hence $\Ee^{\rref}$ is not regular. In section \ref{sec_repre}, we construct a regular Dirichlet form $(\Ee_{\mybar{X}},\Ff_{\mybar{X}})$ on $L^2(\mybar{X};m)$ which is a regular representation of Dirichlet form $(\Ee^\rref,\Ff^\rref_a)$ on $L^2(X;m)$, where $\mybar{X}$ is the Martin compactification of $X$. In section \ref{sec_trace}, we take trace of the regular Dirichlet form $(\Ee_{\mybar{X}},\Ff_{\mybar{X}})$ on $L^2(\mybar{X};m)$ to $K$ to have a regular Dirichlet form $(\Ee_K,\Ff_K)$ on $L^2(K;\nu)$ with the form (\ref{eqn_E_K}). In section \ref{sec_trivial}, we show that $\Ff_K$ consists of constant functions if $\lambda\in(0,1/5)$ or $\beta\in(\beta^*,+\infty)$. Hence $d_w=\beta^*=\log5/\log2$.

\section{Sierpi\'nski Gasket and Sierpi\'nski Graph}\label{sec_SG}

In this section, we review basic constructions of Sierpi\'nski gasket (SG) and Sierpi\'nski graph. SG can be defined in many ways. We give related ones. Let $p_0=(0,0),p_1=(1,0),p_2=(\frac{1}{2},\frac{\sqrt{3}}{2})$, $f_i(x)=(x+p_i)/2$, $x\in\R^2$, $i=0,1,2$. Then SG is the unique nonempty compact set $K$ satisfying $K=f_0(K)\cup f_1(K)\cup f_2(K)$. Let $V_1=\myset{p_0,p_1,p_2}$, $V_{n+1}=f_0(V_n)\cup f_1(V_n)\cup f_2(V_n)$ for all $n\ge1$, then $\myset{V_n}$ is an increasing sequence of finite sets such that $K$ is the closure of $\cup_{n=1}^\infty V_n$. See Figure \ref{figure_gasket}.

\begin{figure}[ht]
  \centering
  \includegraphics[width=0.5\textwidth]{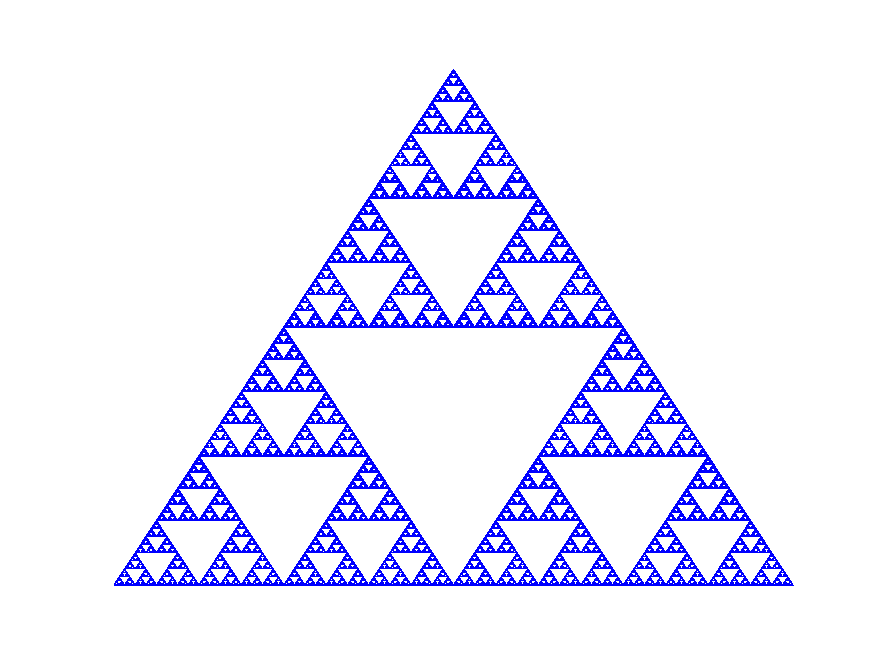}
  \caption{Sierpi\'nski gasket}\label{figure_gasket}
\end{figure}

Let $W_0=\myset{\emptyset}$, $W_n=\myset{w=w_1\ldots w_n:w_i=0,1,2,i=1,\ldots,n}$ for all $n\ge1$ and $W=\cup_{n=0}^\infty W_n$. An element $w=w_1\ldots w_n\in W_n$ is called a finite word with length $n$ and we denote $|w|=n$ for all $n\ge1$. $\emptyset\in W_0$ is called empty word and we denote its length $|\emptyset|=0$, we use the convention that zero length word is empty word. An element in $W$ is called a finite word. Let $W_\infty=\myset{w=w_1w_2\ldots:w_i=0,1,2,i=1,2,\ldots}$ be the set of all infinite sequences with elements in $\myset{0,1,2}$, then an element $w\in W_\infty$ is called an infinite word. For all $w=w_1\ldots w_n\in W$ with $n\ge1$, we write $f_w=f_{w_1}\circ\ldots\circ f_{w_n}$ and we write $f_{\emptyset}=\mathrm{id}$. It is obvious that $K_w=f_w(K)$ is a compact set for all $w\in W$. For all $w=w_1w_2\ldots\in W_\infty$, we write $K_w=\cap_{n=0}^\infty K_{w_1\ldots w_n}$, since $K_{w_1\ldots w_{n+1}}\subseteq K_{w_1\ldots w_n}$ for all $n\ge0$ and $\mathrm{diam}(K_{w_1\ldots w_n})\to0$ as $n\to+\infty$, we have $K_w\subseteq K$ is a one-point set. On the other hand, for all $x\in K$, there exists $w\in W_\infty$ such that $\myset{x}=K_w$. But this $w$ in not unique. For example, for the midpoint $x$ of the segment connecting $p_0$ and $p_1$, we have $\myset{x}=K_{100\ldots}=K_{011\ldots}$, where $100\ldots$ is the element $w=w_1w_2\ldots\in W_\infty$ with $w_1=1,w_n=0$ for all $n\ge2$ and $011\ldots$ has similar meaning.

By representation of infinite words, we can construct Sierpi\'nski graph. First, we construct a triple tree. Take the root $o$ as the empty word $\emptyset$. It has three child nodes, that is, the words in $W_1$, $0,1,2$. Then the nodes $0,1,2$ have child nodes, that is, the words in $W_2$, $0$ has child nodes $00,01,02$, $1$ has child nodes $10,11,12$, $2$ has child nodes $20,21,22$. In general, each node $w_1\ldots w_n$ has three child nodes in $W_{n+1}$, that is, $w_1\ldots w_n0,w_1\ldots w_n1,w_1\ldots w_n2$ for all $n\ge1$. We use node and finite word interchangeable hereafter. For all $n\ge1$ and node $w=w_1\ldots w_n$, the node $w_1\ldots w_{n-1}$ is called the father node of $w$ and denoted by $w^-$. We obtain vertex set $V$ consisting of all nodes. Next, we construct edge set $E$, a subset of $V\times V$. Let
$$
\begin{cases}
E_v&=\myset{(w,w^-),(w^-,w):w\in W_n,n\ge1},\\
E_h&=\myset{(w_1,w_2):w_1,w_2\in W_n,w_1\ne w_2,K_{w_1}\cap K_{w_2}\ne\emptyset,n\ge1},
\end{cases}
$$
and $E=E_v\cup E_h$. $E_v$ is the set of all vertical edges and $E_h$ is the set of all horizontal edges. Then $X=(V,E)$ is Sierpi\'nski graph, see Figure \ref{figure_graph}. We write $X$ for simplicity.



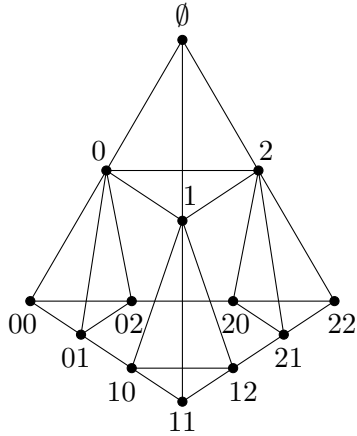
\begin{figure}[ht]
\centering
\begin{tikzpicture}
\draw (0,0)--(4,0);
\draw (0,0)--(2,3.4641016151);
\draw (4,0)--(2,3.4641016151);

\draw (2,3.4641016151)--(2,-1.3333333333);
\draw (4,0)--(2,-1.3333333333);
\draw (0,0)--(2,-1.3333333333);

\draw (1,1.7320508076)--(3,1.7320508076);
\draw (2,1.7320508076-0.6666666667)--(3,1.7320508076);
\draw (1,1.7320508076)--(2,1.7320508076-0.6666666667);

\draw (1,1.7320508076)--(1.3333333333,0);
\draw (1,1.7320508076)--(0.6666666666,-0.4444444444);
\draw (0.6666666666,-0.4444444444)--(1.3333333333,0);

\draw (2,1.7320508076-0.6666666667)--(1.3333333333,-0.8888888888);
\draw (2,1.7320508076-0.6666666667)--(2.6666666666,-0.8888888888);
\draw (2.6666666666,-0.8888888888)--(1.3333333333,-0.8888888888);

\draw (3,1.7320508076)--(3.3333333333,-0.4444444444);
\draw (3,1.7320508076)--(2.6666666666,0);
\draw (2.6666666666,0)--(3.3333333333,-0.4444444444);

\draw[fill=black] (0,0) circle (0.06);
\draw[fill=black] (4,0) circle (0.06);
\draw[fill=black] (2,3.4641016151) circle (0.06);
\draw[fill=black] (2,-1.3333333333) circle (0.06);
\draw[fill=black] (1,1.7320508076) circle (0.06);
\draw[fill=black] (2,1.7320508076-0.6666666667) circle (0.06);
\draw[fill=black] (3,1.7320508076) circle (0.06);
\draw[fill=black] (1.3333333333,0) circle (0.06);
\draw[fill=black] (0.6666666666,-0.4444444444) circle (0.06);
\draw[fill=black] (1.3333333333,-0.8888888888) circle (0.06);
\draw[fill=black] (2.6666666666,-0.8888888888) circle (0.06);
\draw[fill=black] (3.3333333333,-0.4444444444) circle (0.06);
\draw[fill=black] (2.6666666666,0) circle (0.06);

\draw (2,3.8) node {$\emptyset$};
\draw (0.9,2) node {$0$};
\draw (2.1,1.4) node {$1$};
\draw (3.1,2) node {$2$};
\draw (-0.1,-0.3) node {$00$};
\draw (4.1,-0.3) node {$22$};
\draw (1.3,-0.3) node {$02$};
\draw (2.7,-0.3) node {$20$};
\draw (0.6,-0.75) node {$01$};
\draw (3.4,-0.75) node {$21$};
\draw (1.2,-1.2) node {$10$};
\draw (2.8,-1.2) node {$12$};
\draw (2,-1.6) node {$11$};
\end{tikzpicture}
\caption{Sierpi\'nski graph}\label{figure_graph}
\end{figure}


For all $x,y\in V$, if $(x,y)\in E$, then we write $x\sim y$ and say that $y$ is a neighbor of $x$. It is obvious that $\sim$ is an equivalence relation. A path in $X$ is a finite sequence $\pi=[x_0,\ldots,x_n]$ with distinct nodes and $x_0\sim x_1,\ldots,x_{n-1}\sim x_n$, $n$ is called the length of the path. For all $x,y\in V$, let $d(x,y)$ be the graph metric, that is, the minimum length of all paths connecting $x$ and $y$, if a path connecting $x$ and $y$ has length $d(x,y)$, then this path is called geodesic. Hereafter, we write $x\in X$ to mean that $x\in V$. It is obvious that $X$ is a connected and locally finite graph, that is, for all $x,y\in X$ with $x\ne y$, there exists a path connecting $x$ and $y$, for all $x\in X$, the set of its neighbors $\myset{y\in X:x\sim y}$ is a finite set. We write $S_n=\myset{x\in X:|x|=n}$, $B_n=\cup_{i=0}^nS_i$ as sphere and closed ball with radius $n$.

Roughly speaking, for all $n\ge1$, $S_n$ looks like some disconnected triangles, see Figure \ref{figure_S3} for $S_3$, and $V_n$ looks like some connected triangles, see Figure \ref{figure_V3} for $V_3$. We define a mapping $\Phi_n:S_n\to V_n$ as follows. For all $n\ge2$, $w=w_1\ldots w_n\in W_n$, write $p_{w}=p_{w_1\ldots w_n}=f_{w_1\ldots w_{n-1}}(p_{w_n})$. Write $p_1,p_2,p_3$ for $n=1$ and $w=0,1,2$, respectively. By induction, we have $V_n=\cup_{w\in W_n}p_w$ for all $n\ge1$. Define $\Phi_n(w)=p_w$. Then $\Phi_n$ is onto and many pairs of points are mapped into same points, such as $\Phi_3(001)=\Phi_3(010)$. This property can divide the edges in $S_n$ into two types. For an arbitrary edge in $S_n$ with end nodes $x,y$, it is called of type \Rmnum{1} if $\Phi_n(x)\ne\Phi_n(y)$ such as the edge in $S_3$ with end nodes $000$ and $001$, it is called of type \Rmnum{2} if $\Phi_n(x)=\Phi_n(y)$ such as the edge in $S_3$ with end nodes $001$ and $010$. By induction, it is obvious there exist only these two types of edges on each sphere $S_n$.




\begin{figure}[ht]
  \centering
  \begin{tikzpicture}[scale=0.7]
  \tikzstyle{every node}=[font=\small,scale=0.7]
  \draw (0,0)--(9,0);
  \draw (0,0)--(9/2,9/2*1.7320508076);
  \draw (9,0)--(9/2,9/2*1.7320508076);
  
  \draw (3,0)--(3/2,3/2*1.7320508076);
  \draw (6,0)--(15/2,3/2*1.7320508076);
  \draw (3,3*1.7320508076)--(6,3*1.7320508076);
  
  \draw (4,4*1.7320508076)--(5,4*1.7320508076);
  \draw (7/2,7/2*1.7320508076)--(4,3*1.7320508076);
  \draw (11/2,7/2*1.7320508076)--(5,3*1.7320508076);
  
  \draw (1,1*1.7320508076)--(2,1.7320508076);
  \draw (1/2,1/2*1.7320508076)--(1,0);
  \draw (2,0)--(5/2,1/2*1.7320508076);
  
  \draw (7,1.7320508076)--(8,1.7320508076);
  \draw (13/2,1.7320508076/2)--(7,0);
  \draw (17/2,1.7320508076/2)--(8,0);
  
  \draw[fill=black] (0,0) circle (0.06);
  \draw[fill=black] (1,0) circle (0.06);
  \draw[fill=black] (2,0) circle (0.06);
  \draw[fill=black] (3,0) circle (0.06);
  \draw[fill=black] (6,0) circle (0.06);
  \draw[fill=black] (7,0) circle (0.06);
  \draw[fill=black] (8,0) circle (0.06);
  \draw[fill=black] (9,0) circle (0.06);
  \draw[fill=black] (1/2,1.7320508076/2) circle (0.06);
  \draw[fill=black] (5/2,1.7320508076/2) circle (0.06);
  \draw[fill=black] (13/2,1.7320508076/2) circle (0.06);
  \draw[fill=black] (17/2,1.7320508076/2) circle (0.06);
  \draw[fill=black] (1,1.7320508076) circle (0.06);
  \draw[fill=black] (2,1.7320508076) circle (0.06);
  \draw[fill=black] (7,1.7320508076) circle (0.06);
  \draw[fill=black] (8,1.7320508076) circle (0.06);
  \draw[fill=black] (3/2,3/2*1.7320508076) circle (0.06);
  \draw[fill=black] (15/2,3/2*1.7320508076) circle (0.06);
  \draw[fill=black] (3,3*1.7320508076) circle (0.06);
  \draw[fill=black] (4,3*1.7320508076) circle (0.06);
  \draw[fill=black] (5,3*1.7320508076) circle (0.06);
  \draw[fill=black] (6,3*1.7320508076) circle (0.06);
  \draw[fill=black] (7/2,7/2*1.7320508076) circle (0.06);
  \draw[fill=black] (11/2,7/2*1.7320508076) circle (0.06);
  \draw[fill=black] (4,4*1.7320508076) circle (0.06);
  \draw[fill=black] (5,4*1.7320508076) circle (0.06);
  \draw[fill=black] (9/2,9/2*1.7320508076) circle (0.06);
  
  \draw (0,-0.3) node {$000$};
  \draw (1,-0.3) node {$001$};
  \draw (2,-0.3) node {$010$};
  \draw (3,-0.3) node {$011$};
  \draw (6,-0.3) node {$100$};
  \draw (7,-0.3) node {$101$};
  \draw (8,-0.3) node {$110$};
  \draw (9,-0.3) node {$111$};
  
  \draw (0.1,1/2*1.7320508076) node {$002$};
  \draw (2.9,1/2*1.7320508076) node {$012$};
  \draw (6.1,1/2*1.7320508076) node {$102$};
  \draw (8.9,1/2*1.7320508076) node {$112$};
  
  \draw (0.6,1*1.7320508076) node {$020$};
  \draw (2.4,1*1.7320508076) node {$021$};
  \draw (6.6,1*1.7320508076) node {$120$};
  \draw (8.4,1*1.7320508076) node {$121$};
  
  \draw (1.1,3/2*1.7320508076) node {$022$};
  \draw (7.9,3/2*1.7320508076) node {$122$};
  
  \draw (2.6,3*1.7320508076) node {$200$};
  \draw (4,3*1.7320508076-0.3) node {$201$};
  \draw (5,3*1.7320508076-0.3) node {$210$};
  \draw (6.4,3*1.7320508076) node {$211$};
  
  \draw (3.1,7/2*1.7320508076) node {$202$};
  \draw (5.9,7/2*1.7320508076) node {$212$};
  
  \draw (3.6,4*1.7320508076) node {$220$};
  \draw (5.4,4*1.7320508076) node {$221$};
  
  \draw (4.5,9/2*1.7320508076+0.3) node {$222$};
  
  \end{tikzpicture}
  \caption{$S_3$}\label{figure_S3}
\end{figure}






\begin{figure}[ht]
  \centering
  \begin{tikzpicture}[scale=1.125*0.7]
  \tikzstyle{every node}=[font=\small,scale=0.7]
  \draw (0,0)--(8,0);
  \draw (0,0)--(4,4*1.7320508076);
  \draw (8,0)--(4,4*1.7320508076);
  
  \draw (4,0)--(2,2*1.7320508076);
  \draw (2,2*1.7320508076)--(6,2*1.7320508076);
  \draw (6,2*1.7320508076)--(4,0);
  
  \draw (2,0)--(3,1*1.7320508076);
  \draw (3,1*1.7320508076)--(1,1*1.7320508076);
  \draw (1,1*1.7320508076)--(2,0);
  
  \draw (6,0)--(5,1*1.7320508076);
  \draw (5,1*1.7320508076)--(7,1*1.7320508076);
  \draw (7,1*1.7320508076)--(6,0);
  
  \draw (4,2*1.7320508076)--(3,3*1.7320508076);
  \draw (3,3*1.7320508076)--(5,3*1.7320508076);
  \draw (5,3*1.7320508076)--(4,2*1.7320508076);
  
  \draw[fill=black] (0,0) circle (0.06);
  \draw[fill=black] (2,0) circle (0.06);
  \draw[fill=black] (4,0) circle (0.06);
  \draw[fill=black] (6,0) circle (0.06);
  \draw[fill=black] (8,0) circle (0.06);
  \draw[fill=black] (1,1.7320508076) circle (0.06);
  \draw[fill=black] (3,1.7320508076) circle (0.06);
  \draw[fill=black] (5,1.7320508076) circle (0.06);
  \draw[fill=black] (7,1.7320508076) circle (0.06);
  \draw[fill=black] (2,2*1.7320508076) circle (0.06);
  \draw[fill=black] (4,2*1.7320508076) circle (0.06);
  \draw[fill=black] (6,2*1.7320508076) circle (0.06);
  \draw[fill=black] (3,3*1.7320508076) circle (0.06);
  \draw[fill=black] (5,3*1.7320508076) circle (0.06);
  \draw[fill=black] (4,4*1.7320508076) circle (0.06);
  
  \draw (0,-0.3) node {$p_{000}$};
  \draw (2,-0.3) node {$p_{001}=p_{010}$};
  \draw (4,-0.3) node {$p_{011}=p_{100}$};
  \draw (6,-0.3) node {$p_{101}=p_{110}$};
  \draw (8,-0.3) node {$p_{111}$};
  \draw (4,4*1.7320508076+0.3) node {$p_{222}$};
  
  \draw (0,1*1.7320508076) node {$p_{002}=p_{020}$};
  \draw (8,1*1.7320508076) node {$p_{112}=p_{121}$};
  \draw (2.95,1*1.7320508076+0.3) node {$p_{012}=p_{021}$};
  \draw (5.05,1*1.7320508076+0.3) node {$p_{102}=p_{120}$};
  
  \draw (1,2*1.7320508076) node {$p_{022}=p_{200}$};
  \draw (7,2*1.7320508076) node {$p_{122}=p_{211}$};
  
  \draw (4,2*1.7320508076-0.3) node {$p_{201}=p_{210}$};
  
  \draw (2,3*1.7320508076) node {$p_{202}=p_{220}$};
  \draw (6,3*1.7320508076) node {$p_{212}=p_{221}$};
  \end{tikzpicture}
  \caption{$V_3$}\label{figure_V3}
\end{figure}



Sierpi\'nski graph is a hyperbolic graph, see \cite[Theorem 3.2]{LW09}. For arbitrary graph $X$, choose a node $o$ as root, define graph metric $d$ as above, write $|x|=d(o,x)$. For all $x,y\in X$, define Gromov product
$$|x\wedge y|=\frac{1}{2}(|x|+|y|-d(x,y)).$$
$X$ is called a hyperbolic graph if there exists $\delta>0$ such that for all $x,y,z\in X$, we have
$$|x\wedge y|\ge\mathrm{min}{\myset{|x\wedge z|,|z\wedge y|}}-\delta.$$
It is known that the definition is independent of the choice of root $o$. For a hyperbolic graph, we can introduce a metric as follows. Choose $a>0$ such that $a'=e^{3\delta a}-1<\sqrt{2}-1$. For all $x,y\in X$, define
$$
\rho_a(x,y)=
\begin{cases}
\exp{(-a|x\wedge y|)},&\text{if }x\ne y,\\
0,&\text{if }x=y,
\end{cases}
$$
then $\rho_a$ satisfies
$$\rho_a(x,y)\le(1+a')\max{\myset{\rho_a(x,z),\rho_a(z,y)}}\text{ for all }x,y,z\in X.$$
This means $\rho_a$ is an ultra-metric not a metric. But we can define
$$\theta_a(x,y)=\inf{\myset{\sum_{i=1}^n\rho_a(x_{i-1},x_i):x=x_0,\ldots,x_n=y,x_i\in X,i=0,\ldots,n,n\ge1}},$$
for all $x,y\in X$. $\theta_a$ is a metric and equivalent to $\rho_a$. So we use $\rho_a$ rather than $\theta_a$ for simplicity. It is known that a sequence $\myset{x_n}\subseteq X$ with $|x_n|\to+\infty$ is a Cauchy sequence in $\rho_a$ if and only if $|x_m\wedge x_n|\to+\infty$ as $m,n\to+\infty$. Let $\hat{X}$ be the completion of $X$ with respect to $\rho_a$, then $\dd_hX=\hat{X}\backslash X$ is called the hyperbolic boundary of $X$. By \cite[Corollary 22.13]{Wo00}, $\hat{X}$ is compact. It is obvious that hyperbolicity is only related to the graph structure of $X$. We introduce a description of hyperbolic boundary in terms of geodesic rays. A geodesic ray is a sequence $[x_0,x_1,\ldots]$ with distinct nodes, $x_n\sim x_{n+1}$ and path $[x_0,\ldots,x_n]$ is geodesic for all $n\ge0$. Two geodesic rays $\pi=[x_0,x_1,\ldots]$ and $\pi'=[y_0,y_1,\ldots]$ are called equivalent if $\varliminf_{n\to+\infty}d(y_n,\pi)<+\infty$, where $d(x,\pi)=\inf_{n\ge0}d(x,x_n)$. There exists a one-to-one correspondence between the family of all equivalent geodesic rays and hyperbolic boundary as follows.

By \cite[Proposition 22.12(b)]{Wo00}, equivalence geodesic rays is an equivalence relation. By \cite[Lemma 22.11]{Wo00}, for all geodesic ray $\pi=[x_0,x_1,\ldots]$, for all $u\in X$, there exist $k,l\ge0$, $u=u_0,\ldots,u_k=x_l$, such that, $[u,u_1,\ldots,u_k,x_{l+1},x_{l+2},\ldots]$ is a geodesic ray. It is obvious that this new geodesic ray is equivalent to $\pi$, hence we can take a geodesic ray in each equivalence class of the form $\pi=[x_0,x_1,\ldots]$, $|x_n|=n$, $x_n\sim x_{n+1}$ for all $n\ge0$. By \cite[Proposition 22.12(c)]{Wo00}, we can define a one-to-one mapping $\tau$ from the family of all equivalent geodesic rays to hyperbolic boundary, $\tau:[x_0,x_1,\ldots]\mapsto\text{the limit }\xi\text{ of }\myset{x_n}\text{ in }\rho_a$. By above, we can choose $[x_0,x_1,\ldots]$ of the form $|x_n|=n$, $x_n\sim x_{n+1}$ for all $n\ge0$, we say that $[x_0,x_1,\ldots]$ is a geodesic ray from $o$ to $\xi$.

For $y\in\hat{X}$, $x\in X$, we say that $y$ is in the subtree with root $x$ if $x$ lies on the geodesic path or some geodesic ray from $o$ to $y$. And if $y$ is in the subtree with root $x$, then it is obvious that $|x\wedge y|=|x|$, $\rho_a(x,y)=e^{-a|x|}$ if $x\ne y$. For more detailed discussion of hyperbolic graph, see \cite[Chapter \Rmnum{4}, \Rmnum{4}.22]{Wo00}.

\cite[Theorem 3.2, Theorem 4.3, Proposition 4.4]{LW09} showed that for a general class of fractals satisfying open set condition (OSC), we can construct an augmented rooted tree which is a hyperbolic graph and the hyperbolic boundary is H\"older equivalent to the fractal through canonical mapping. In particular, SG satisfies OSC, Sierpi\'nski graph is an augmented rooted tree hence hyperbolic. The canonical mapping $\Phi$ can be described as follows.

For all $\xi\in\dd_hX$, there corresponds a geodesic ray in the equivalence class corresponding to $\xi$ through the mapping $\tau$ of the form $[x_0,x_1,\ldots]$ with $|x_n|=n$ and $x_n\sim x_{n+1}$ for all $n\ge0$, then there exists an element $w\in W_\infty$ such that $w_1\ldots w_n=x_n$ for all $n\ge1$. Then $\myset{\Phi(\xi)}=K_w$ and
\begin{equation}\label{eqn_Holder}
\lvert\Phi(\xi)-\Phi(\eta)\rvert\asymp\rho_a(\xi,\eta)^{\log2/a}\text{ for all }\xi,\eta\in\dd_hX.
\end{equation}

\section{Random Walk on $X$}\label{sec_rw}

In this section, we give a transient reversible random walk on $X$ from \cite{KLW17}. Let $c:X\times X\to[0,+\infty)$ be conductance satisfying
$$
\begin{aligned}
c(x,y)&=c(y,x),\\
\pi(x)&=\sum_{y\in X}c(x,y)\in(0,+\infty),\\
c(x,y)&>0\text{ if and only if }x\sim y,
\end{aligned}
$$
for all $x,y\in X$. Let $P(x,y)=c(x,y)/\pi(x)$, $x,y\in X$, then $P$ is a transition probability satisfying $\pi(x)P(x,y)=\pi(y)P(y,x)$ for all $x,y\in X$. We construct a reversible random walk $Z=\myset{Z_n}$ on $X$ with transition probability $P$. We introduce some related quantities. Let $P^{(0)}(x,y)=\delta_{xy}$, $P^{(n+1)}(x,y)=\sum_{z\in X}P(x,z)P^{(n)}(z,y)$ for all $x,y\in X$, $n\ge0$. Define
$$G(x,y)=\sum_{n=0}^\infty P^{(n)}(x,y),x,y\in X,$$
then $G$ is the Green function of $Z$ and $Z$ is called transient if $G(x,y)<+\infty$ for all or equivalent for some $x,y\in X$. Define
$$F(x,y)=\PP_x\left[Z_n=y\text{ for some }n\ge0\right],$$
that is, the probability of ever reaching $y$ starting from $x$. By Markov property, we have
$$G(x,y)=F(x,y)G(y,y).$$
For more detailed discussion of general theory of random walk, see \cite[Chapter \Rmnum{1}, \Rmnum{1}.1]{Wo00}.

Here, we take some specific random walk called $\lambda$-return ratio random walk introduced in \cite{KLW17}, that is,
$$\frac{c(x,x^-)}{\sum_{y:y^-=x}c(x,y)}=\frac{P(x,x^-)}{\sum_{y:y^-=x}P(x,y)}=\lambda\in(0,+\infty)\text{ for all }x\in X\text{ with }|x|\ge1.$$
For all $n\ge0$, $x\in S_n,y\in S_{n+1}$, we take $c(x,y)$ the same value denoted by $c(n,n+1)=c(n+1,n)$. Then
$$\lambda=\frac{c(n-1,n)}{3c(n,n+1)},$$
that is,
$$c(n,n+1)=\frac{c(n-1,n)}{3\lambda}=\ldots=\frac{1}{(3\lambda)^n}{c(0,1)}.$$
Take $c(0,1)=1$, then $c(n,n+1)=1/(3\lambda)^n$. Moreover, \cite[Definition 4.4]{KLW17} gave restrictions to conductance of horizontal edges. For all $n\ge1$, $x,y\in S_n$, $x\sim y$, let
$$
c(x,y)=\\
\begin{cases}
\frac{C_1}{(3\lambda)^n},&\text{ if the edge with end nodes }x,y\text{ is of type \Rmnum{1}},\\
\frac{C_2}{(3\lambda)^n},&\text{ if the edge with end nodes }x,y\text{ is of type \Rmnum{2}},
\end{cases}
$$
where $C_1,C_2$ are some positive constants.

\cite[Proposition 4.1, Lemma 4.2]{KLW17} showed that if $\lambda\in(0,1)$, then $Z$ is transient and
\begin{equation}\label{eqn_G}
G(o,o)=\frac{1}{1-\lambda},
\end{equation}
\begin{equation}\label{eqn_F}
F(x,0)=\lambda^{|x|}\text{ for all }x\in X.
\end{equation}

For a transient random walk, we can introduce Martin kernel given by
$$K(x,y)=\frac{G(x,y)}{G(o,y)},$$
and Martin compactification $\mybar{X}$, that is, the smallest compactification such that $K(x,\cdot)$ can be extended continuously for all $x\in X$. Martin boundary is given by $\dd_MX=\mybar{X}\backslash X$. Then Martin kernel $K$ can be defined on $X\times\mybar{X}$.

\cite[Theorem 5.1]{KLW17} showed that the Martin boundary $\dd_MX$, the hyperbolic boundary $\dd_hX$ and SG $K$ are homeomorphic. Hence the completion $\hat{X}$ of $X$ with respect to $\rho_a$ and Martin compactification $\mybar{X}$ are homeomorphic. It is always convenient to consider $\hat{X}$ rather than $\mybar{X}$. We use $\dd X$ to denote all these boundaries. We list some general results of Martin boundary for later use.

\begin{mythm}\label{thm_conv}(\cite[Theorem 24.10]{Wo00})
Let $Z$ be transient, then $\myset{Z_n}$ converges to a $\dd_MX$-valued random variable $Z_\infty$, $\PP_x$-a.s. for all $x\in X$. The hitting distribution of $\myset{Z_n}$ or the distribution of $Z_\infty$ under $\PP_x$, denoted by $\nu_x$, satisfies
$$\nu_x(B)=\int_BK(x,\cdot)\md\nu_o\text{ for all Borel measurable set }B\subseteq\dd_M X,$$
that is, $\nu_x$ is absolutely continuous with respect to $\nu_o$ with Radon-Nikodym derivative $K(x,\cdot)$. 
\end{mythm}

For all $\nu_o$-integrable function $\vphi$ on $\dd_MX$, we have
$$h(x)=\int_{\dd_MX}\vphi\md\nu_x=\int_{\dd_MX}K(x,\cdot)\vphi\md\nu_o,x\in X,$$
is a harmonic function on $X$. It is called the Poisson integral of $\vphi$, denoted by $H\vphi$.

\cite[Theorem 5.6]{KLW17} showed that the hitting distribution $\nu_o$ is normalized Hausdorff measure on $K$. We write $\nu$ for $\nu_o$ for simplicity.

Using conductance $c$, we construct an energy on $X$ given by
$$\Ee_X(u,u)=\frac{1}{2}\sum_{x,y\in X}c(x,y)(u(x)-u(y))^2.$$

In \cite{Sil74}, Silverstein constructed Na\"im kernel $\Theta$ on $\mybar{X}\times\mybar{X}$ using Martin kernel to induce an energy on $\dd X$ given by
$$\Ee_{\dd X}(u,u)=\Ee_X(Hu,Hu)=\frac{1}{2}\pi(o)\int_{\dd X}\int_{\dd X}(u(x)-u(y))^2\Theta(x,y)\nu(\md x)\nu(\md y),$$
for all $u\in L^2(\dd_MX;\nu)$ with $\Ee_{\dd X}(u,u)<+\infty$.

\cite[Theorem 6.3]{KLW17} calculated Na\"im kernel forcefully
\begin{equation}\label{eqn_Theta}
\Theta(x,y)\asymp\frac{1}{|x-y|^{\alpha+\beta}},
\end{equation}
where $\alpha=\log3/\log2$ is Hausdorff dimension of SG, $\beta=-\log\lambda/\log2\in(0,+\infty)$, $\lambda\in(0,1)$. No message of upper bound for $\beta$ of walk dimension appeared in their calculation.

\section{Regular Dirichlet Form on $X$}\label{sec_df_X}

In this section, we construct a regular Dirichlet form $\Ee_X$ on $X$ and its corresponding symmetric Hunt process $\myset{X_t}$. We prove that the Martin boundaries of $\myset{X_t}$ and $Z$ coincide. We show that $\Ee_X$ is stochastically incomplete and $\myset{X_t}$ goes to infinity in finite time almost surely.

Let $m:X\to(0,+\infty)$ be a positive function given by
$$m(x)=\left(\frac{c}{3\lambda}\right)^{|x|},x\in X,$$
where $c\in(0,\lambda)\subseteq(0,1)$. Then $m$ can be regarded as a measure on $X$. Note that
$$m(X)=\sum_{x\in X}m(x)=\sum_{n=0}^\infty3^n\cdot\left(\frac{c}{3\lambda}\right)^n=\sum_{n=0}^{\infty}\left(\frac{c}{\lambda}\right)^n<+\infty,$$
we have $m$ is a finite measure on $X$. We construct a symmetric form on $L^2(X;m)$ given by
$$
\begin{cases}
&\Ee_X(u,u)=\frac{1}{2}\sum_{x,y\in X}c(x,y)(u(x)-u(y))^2,\\
&\Ff_X=\text{the }(\Ee_X)_1\text{-closure of }C_0(X),
\end{cases}
$$
where $C_0(X)$ is the set of all functions with finite support. It is obvious that $(\Ee_X,\Ff_X)$ is a regular Dirichlet form on $L^2(X;m)$. By \cite[Theorem 7.2.1]{FOT11}, it corresponds to a symmetric Hunt process on $X$. Roughly speaking, this process is a variable speed continuous time random walk characterized by holding at one node with time distributed to exponential distribution and jumping according to random walk. For some discussion of continuous time random walk, see \cite[Chapter 2]{Nor98}. We give detailed construction as follows.

Let $(\Omega,\Ff,\PP)$ be a probability space on which given a random walk $\myset{Y_n}$ with transition probability $P$ and initial distribution $\sigma$ and a sequence of independent exponential distributed random variables $\myset{S_n}$ with parameter $1$, that is, $\PP[S_n\in\md t]=e^{-t}\md t$. Assume that $\myset{S_n}$ is independent of $\myset{Y_n}$. Let $\alpha(x)=\pi(x)/m(x)$, $x\in X$. For all $n\ge1$, let $T_n=S_n/\alpha(Y_{n-1})$, $J_n=T_1+\ldots+T_n$, $J_0=0$. Then $T_n$ is called the $n$-th holding time and $J_n$ is called the $n$-th jumping time. Let
$$
X_t=
\begin{cases}
Y_n,&\text{if }J_n\le t<J_{n+1}\text{ for some }n\ge0,\\
\dd,&\text{otherwise},
\end{cases}
$$
where $\dd$ is death point. This construction is similar to that of Poisson process and it is called variable speed continuous time random walk in some literature. $\myset{X_t}$ is a symmetric Hunt process with initial distribution $\sigma$. By calculating the generators of $\Ee_X$ and $\myset{X_t}$, we have $\myset{X_t}$ is the symmetric Hunt process corresponding to $\Ee_X$.

By the construction of $\myset{X_t}$ in terms of $\myset{Y_n}$, the Martin boundary of $\myset{X_t}$ is the same as the Martin boundary of $Z$. Indeed, by Dirichlet form theory, the Green function of $\myset{X_t}$ is given by
$$g(x,y)=\frac{G(x,y)}{\pi(y)}\text{ for all }x,y\in X.$$
Hence the Martin kernel of $\myset{X_t}$ is given by
$$k(x,y)=\frac{g(x,y)}{g(o,y)}=\frac{G(x,y)/\pi(y)}{G(o,y)/\pi(y)}=\frac{G(x,y)}{G(o,y)}=K(x,y)\text{ for all }x,y\in X.$$
Hence the Martin boundaries of $\myset{X_t}$ and $Z$ coincide. Moreover, $\Ee_X$ is transient.

\begin{mythm}\label{thm_sto}
$(\Ee_X,\Ff_X)$ on $L^2(X;m)$ is stochastically incomplete.
\end{mythm}

We prove stochastic incompleteness by considering lifetime $\zeta=\sum_{n=1}^\infty T_n=\lim_{n\to+\infty}J_n$. This quantity is called the (first) explosion time in \cite[Chapter 2, 2.2]{Nor98}. We need a proposition for preparation.

\begin{myprop}\label{prop_jump}
The jumping times $J_n$ are stopping times of $\myset{X_t}$ for all $n\ge0$.
\end{myprop}

\begin{proof}
Let $\myset{\Ff_t}$ be the minimum completed admissible filtration with respect to $\myset{X_t}$. It is obvious that $J_0=0$ is a stopping time of $\myset{X_t}$. Assume that $J_n$ is a stopping time of $\myset{X_t}$, then for all $t\ge0$, we have
$$\myset{J_{n+1}\le t}=\cup_{s\in\mathbb{Q},s\le t}\left(\myset{J_n\le t}\cap\myset{X_s\ne X_{J_n}}\right)\in\Ff_t,$$
hence $J_{n+1}$ is a stopping time of $\myset{X_t}$. By induction, it follows that $J_n$ are stopping times of $\myset{X_t}$ for all $n\ge0$.
\end{proof}

\begin{proof}[Proof of Theorem \ref{thm_sto}]
By Equation (\ref{eqn_F}), we have
$$
\begin{aligned}
\EE_o\zeta&=\EE_o\sum_{n=1}^\infty T_n=\sum_{n=1}^\infty\EE_o\left[\frac{S_n}{\alpha(Y_{n-1})}\right]=\sum_{n=1}^\infty\EE_o[{S_n}]\EE_o\left[\frac{1}{\alpha(Y_{n-1})}\right]=\sum_{n=1}^\infty\EE_o\frac{m}{\pi}(Y_{n-1})\\
&=\sum_{n=0}^\infty\EE_o\frac{m}{\pi}(Y_{n})=\sum_{n=0}^\infty\sum_{x\in X}\frac{m(x)}{\pi(x)}P^{(n)}(o,x)=\sum_{x\in X}\frac{m(x)}{\pi(x)}G(o,x)\\
&=\sum_{x\in X}\frac{m(x)}{\pi(x)}\frac{\pi(x)G(x,o)}{\pi(o)}=\sum_{x\in X}\frac{m(x)G(x,o)}{\pi(o)}=\sum_{x\in X}\frac{m(x)F(x,o)G(o,o)}{\pi(o)}\\
&=\frac{G(o,o)}{\pi(o)}\sum_{n=0}^\infty3^n\cdot\left(\frac{c}{3\lambda}\right)^n\cdot\lambda^n=\frac{G(o,o)}{\pi(o)}\sum_{n=0}^\infty c^n.
\end{aligned}
$$
Since $c\in(0,\lambda)\subseteq(0,1)$, we have $\EE_o\zeta<+\infty$, $\PP_o[\zeta<+\infty]=1$.

For all $x\in X$, let $n=|x|$, note that $P^{(n)}(o,x)>0$, by Proposition \ref{prop_jump} and strong Markov property, we have
$$
\begin{aligned}
\EE_o\zeta&\ge\EE_o\left[\zeta\indi_{\myset{X_{J_n}=x}}\right]=\EE_o\left[\EE_o\left[\zeta\indi_{\myset{X_{J_n}=x}}|X_{J_n}\right]\right]=\EE_o\left[\indi_{\myset{X_{J_n}=x}}\EE_o\left[\zeta|X_{J_n}\right]\right]\\
&=\EE_o\left[\indi_{\myset{X_{J_n}=x}}\EE_{X_{J_n}}\left[\zeta\right]\right]=P^{(n)}(o,x)\EE_x\left[\zeta\right].
\end{aligned}
$$
Hence $\EE_x\zeta<+\infty$, $\PP_x\left[\zeta<+\infty\right]=1$ for all $x\in X$, $\Ee_X$ is stochastically incomplete.
\end{proof}

By \cite[Proposition 1.17(b)]{Wo00}, for a transient random walk $Z$ on $X$, for all finite set $A\subseteq X$, we have $\PP_x\left[Z_n\in A\text{ for infinitely many }n\right]=0$ for all $x\in X$. Roughly speaking, a transient random walk will go to infinity almost surely. For variable speed continuous time random walk $\myset{X_t}$ on $X$, we have following theorem.

\begin{mythm}\label{thm_infty}
$\myset{X_t}$ goes to infinity in finite time almost surely, that is,
$$\PP_x\left[\lim_{t\uparrow\zeta}\lvert X_t\rvert=+\infty,\zeta<+\infty\right]=1\text{ for all }x\in X.$$
\end{mythm}
\begin{proof}
There exists $\Omega_0$ with $\PP_x(\Omega_0)=1$ such that $\zeta(\omega)<+\infty$ for all $\omega\in\Omega_0$. For all $m\ge1$, we have $\PP_x\left[Y_n\in B_m\text{ for infinitely many }n\right]=0$, hence there exists $\Omega_m$ with $\PP_x(\Omega_m)=1$ such that for all $\omega\in\Omega_m$, there exist $N=N(\omega)\ge1$, for all $n\ge N$, $Y_n(\omega)\notin B_m$. Hence $\PP_x\left(\Omega_0\cap\cap_{m=1}^\infty\Omega_m\right)=1$. For all $\omega\in\Omega_0\cap\cap_{m=1}^\infty\Omega_m$, we have $J_n(\omega)\le J_{n+1}(\omega)<\zeta(\omega)<+\infty$. For all $m\ge1$, since $\omega\in\Omega_m$, there exists $N=N(\omega)\ge1$, for all $n>N$, we have $Y_n(\omega)\notin B_m$. By definition, $X_t(\omega)=Y_n(\omega)$ if $J_n(\omega)\le t<J_{n+1}(\omega)$. Letting $T=J_{N(\omega)}(\omega)$, for all $t>T$ there exists $n\ge N$ such that $J_n(\omega)\le t<J_{n+1}(\omega)$, hence $X_t(\omega)=Y_n(\omega)\not\in B_m$, that is, $\lim_{t\uparrow\zeta(\omega)}\lvert X_t(\omega)\rvert=+\infty$. We obtain the desired result.
\end{proof}

\section{Active Reflected Dirichlet Space $(\Ee^{\rref},\Ff^{\rref}_a)$}\label{sec_ref}

In this section, we construct active reflected Dirichlet form $(\Ee^{\rref},\Ff^{\rref}_a)$ and show that $\Ff_X\subsetneqq\Ff^{\rref}_a$, hence $\Ee^{\rref}$ is not regular.

Reflected Dirichlet space was introduced by Chen \cite{Chen92}. This is a generalization of reflected Brownian motion in Euclidean space. He considered abstract Dirichlet form instead of constructing reflection path-wisely from probabilistic viewpoint. More detailed discussion is incorporated into his book with Fukushima \cite[Chapter 6]{CF12}.

Given a regular transient Dirichlet form $(\Ee,\Ff)$ on $L^2(X;m)$, we can do reflection in two ways:
\begin{itemize}
\item The linear span of $\Ff$ and all harmonic functions of finite ``$\Ee$-energy".
\item All functions that are ``locally" in $\Ff$ and have finite ``$\Ee$-energy".
\end{itemize}
We use the second way which is more convenient. Recall Beurling-Deny decomposition. Since $(\Ee,\Ff)$ is regular, we have
$$\Ee(u,u)=\frac{1}{2}\mu^c_{<u>}(X)+\int_{X\times X\backslash d}(u(x)-u(y))^2J(\md x\md y)+\int_Xu(x)^2k(\md x)$$
for all $u\in\Ff_e$, here we use the convention that all functions in $\Ff_e$ are quasi-continuous. By this formula, we can define
$$\hat{\Ee}(u,u)=\frac{1}{2}\mu^c_{<u>}(X)+\int_{X\times X\backslash d}(u(x)-u(y))^2J(\md x\md y)+\int_Xu(x)^2k(\md x)$$
for all $u\in\Ff_{\mathrm{loc}}$. We give the definition of reflected Dirichlet space as follows. \cite[Theorem 6.2.5]{CF12} gave
$$
\begin{cases}
&\Ff^{\rref}=\myset{u:\text{ finite }m\text{-a.e.},\exists\myset{u_n}\subseteq\Ff_{\mathrm{loc}}\text{ that is }\hat{\Ee}\text{-Cauchy},u_n\to u,m\text{-a.e. on }X},\\
&\hat{\Ee}(u,u)=\lim_{n\to+\infty}\hat{\Ee}(u_n,u_n).
\end{cases}
$$
Let $\tau_ku=\left((-k)\vee u\right)\wedge k$, $k\ge1$, then \cite[Theorem 6.2.13]{CF12} gave
$$
\begin{cases}
&\Ff^{\rref}=\myset{u:|u|<+\infty,m\text{-a.e.},\tau_ku\in\Ff_{\mathrm{loc}}\forall k\ge1,\sup_{k\ge1}\hat{\Ee}(\tau_ku,\tau_ku)<+\infty},\\
&\Ee^\rref(u,u)=\lim_{k\to+\infty}\hat{\Ee}(\tau_ku,\tau_ku).
\end{cases}
$$
Let $\Ff^\rref_a=\Ff^\rref\cap L^2(X;m)$, then $(\Ff^\rref_a,\Ee^\rref)$ is called active reflected Dirichlet space. \cite[Theorem 6.2.14]{CF12} showed that $(\Ee^\rref,\Ff^\rref_a)$ is a Dirichlet form on $L^2(X;m)$.

Return to our case, since
$$\Ee_X(u,u)=\frac{1}{2}\sum_{x,y\in X}c(x,y)(u(x)-u(y))^2\text{ for all }u\in\Ff_X,$$
$\Ee_X$ has only jumping part, we have
$$\hat{\Ee}_X(u,u)=\frac{1}{2}\sum_{x,y\in X}c(x,y)(u(x)-u(y))^2\text{ for all }u\in(\Ff_X)_{\mathrm{loc}}.$$
By the definition of local Dirichlet space
$$(\Ff_X)_{\mathrm{loc}}=\myset{u:\forall G\subseteq X\text{ relatively compact open,}\exists v\in\Ff_X,\text{s.t. }u=v,m\text{-a.e. on }G}.$$
For all $G\subseteq X$ relatively compact open, we have $G$ is a finite set, for all function $u$ on $X$, let $v(x)=u(x)$, $x\in G$, $v(x)=0$, $x\in X\backslash G$, then $v\in C_0(X)\subseteq\Ff_X$ and $u=v$ on $G$, hence $(\Ff_X)_{\mathrm{loc}}=\myset{u:u\text{ is a finite function on }X}$.
$$\Ff^\rref=\myset{u:|u(x)|<+\infty,\forall x\in X,\sup_{k\ge1}\frac{1}{2}\sum_{x,y\in X}c(x,y)\left(\tau_ku(x)-\tau_ku(y)\right)^2<+\infty}.$$
By monotone convergence theorem, we have
$$\frac{1}{2}\sum_{x,y\in X}c(x,y)\left(\tau_ku(x)-\tau_ku(y)\right)^2\uparrow\frac{1}{2}\sum_{x,y\in X}c(x,y)\left(u(x)-u(y)\right)^2,$$
hence
$$\sup_{k\ge1}\frac{1}{2}\sum_{x,y\in X}c(x,y)\left(\tau_ku(x)-\tau_ku(y)\right)^2=\frac{1}{2}\sum_{x,y\in X}c(x,y)\left(u(x)-u(y)\right)^2,$$
and
$$
\begin{cases}
&\Ff^\rref=\myset{u:\text{ finite function},\frac{1}{2}\sum_{x,y\in X}c(x,y)\left(u(x)-u(y)\right)^2<+\infty},\\
&\Ee^\rref(u,u)=\frac{1}{2}\sum_{x,y\in X}c(x,y)\left(u(x)-u(y)\right)^2.
\end{cases}
$$
$$\Ff^\rref_a=\myset{u\in L^2(X;m):\frac{1}{2}\sum_{x,y\in X}c(x,y)\left(u(x)-u(y)\right)^2<+\infty}.$$
Indeed, we can show that $(\Ee^\rref,\Ff^\rref_a)$ is a Dirichlet form on $L^2(X;m)$ directly. In general, $(\Ee^\rref,\Ff^\rref_a)$ on $L^2(X;m)$ is not regular, $\Ff_X\subsetneqq\Ff^\rref_a$. This is like $H^1_0(D)\subsetneqq H^1(D)$. We need to show $\Ff_X\ne\Ff^\rref_a$, otherwise reflection is meaningless. Then we do regular representation of $(\Ee^\rref,\Ff^\rref_a)$ on $L^2(X;m)$ to enlarge the space $X$ to Martin compactification $\mybar{X}$ and Martin boundary $\dd X$ will appear.

\begin{mythm}
$\Ff_X\subsetneqq\Ff^{\rref}_a$, hence $\Ee^{\rref}$ is not regular.
\end{mythm}

\begin{proof}
Since $m(X)<+\infty$, we have $1\in\Ff^{\rref}_a$ and $\Ee^{\rref}(1,1)=0$, by \cite[Theorem 1.6.3]{FOT11}, $\Ee^{\rref}$ is recurrent, by \cite[Lemma 1.6.5]{FOT11}, $\Ee^{\rref}$ is conservative or stochastically complete. Since $\Ee_X$ is transient and stochastically incomplete, we have $\Ff_X\ne\Ff^{\rref}_a$. Note that $\Ee^{\rref}$ is not regular, there is no corresponding Hunt process, but recurrent and conservative properties are still well-defined, see \cite[Chapter 1, 1.6]{FOT11}.
\end{proof}

\section{Regular Representation of $(\Ee^\rref,\Ff^\rref_a)$}\label{sec_repre}

In this section, we construct a regular Dirichlet form $(\Ee_{\mybar{X}},\Ff_{\mybar{X}})$ on $L^2(\mybar{X};m)$ which is a regular representation of Dirichlet form $(\Ee^\rref,\Ff^\rref_a)$ on $L^2(X;m)$, where $\mybar{X}$ is the Martin compactification of $X$ and $m$ is given as above.

Recall that $(\frac{1}{2}\DD,H^1(D))$ on $L^2(D)$ is not regular and $(\frac{1}{2}\DD,H^1(D))$ on $L^2(\mybar{D},\indi_D(\md x))$ is a regular representation. Our construction is very simple and similar to this case. Let
$$
\begin{cases}
&\Ee_{\mybar{X}}(u,u)=\frac{1}{2}\sum_{x,y\in X}c(x,y)(u(x)-u(y))^2,\\
&\Ff_{\mybar{X}}=\myset{u\in C(\mybar{X}):\sum_{x,y\in X}c(x,y)(u(x)-u(y))^2<+\infty}.
\end{cases}
$$
We show that $(\Ee_{\mybar{X}},\Ff_{\mybar{X}})$ is a regular Dirichlet form on $L^2({\mybar{X}};m)$.

\begin{mythm}\label{thm_main}
If $\lambda\in(1/5,1/3)$, then $(\Ee_{\mybar{X}},\Ff_{\mybar{X}})$ is a regular Dirichlet form on $L^2({\mybar{X}};m)$.
\end{mythm}

First, we need a lemma.

\begin{mylem}\label{lem_ext}
If $\lambda<1/3$, then for all $u$ on $X$ with
$$C=\frac{1}{2}\sum_{x,y\in X}c(x,y)(u(x)-u(y))^2<+\infty,$$
$u$ can be extended continuously to $\mybar{X}$.
\end{mylem}

\begin{proof}
Since $\hat{X}$ is homeomorphic to $\mybar{X}$, we consider $\hat{X}$ instead. For all $\xi\in\dd X$, take geodesic ray $[x_0,x_1,\ldots]$ with $|x_n|=n$, $x_n\sim x_{n+1}$ for all $n\ge0$ such that $x_n\to\xi$ in $\rho_a$. Then
$$\lvert u(x_n)-u(x_{n+1})\rvert\le\sqrt{\frac{2C}{c(x_n,x_{n+1})}}=\sqrt{2C}(\sqrt{3\lambda})^n,$$
since $\lambda<1/3$, we have $\myset{u(x_n)}$ is a Cauchy sequence, define $u(\xi)=\lim_{n\to+\infty}u(x_n)$.

First, we show that this is well-defined. Indeed, for all equivalent geodesic rays $[x_0,x_1,\ldots]$ and $[y_0,y_1,\ldots]$ with $|x_0|=|y_0|=0$, by \cite[Proposition 22.12(a)]{Wo00}, for all $n\ge0$, $d(x_n,y_n)\le2\delta$. Take an integer $M\ge 2\delta$, then for arbitrary fixed $n\ge0$, there exist $z_0=x_n,\ldots,z_M=y_n$ with $|z_i|=n$ for all $i=0,\ldots,M$, $z_i=z_{i+1}$ or $z_{i}\sim z_{i+1}$ for all $i=0,\ldots, M-1$, we have
$$\lvert u(x_n)-u(y_n)\rvert\le\sum_{i=0}^{M-1}|u(z_{i})-u(z_{i+1})|\le\sum_{i=0}^{M-1}\sqrt{\frac{2C}{c(z_i,z_{i+1})}}\le M\sqrt{\frac{2C}{\min{\myset{C_1,C_2}}}}(\sqrt{3\lambda})^n.$$
Since $\lambda<1/3$, letting $n\to+\infty$, we have $\lvert u(x_n)-u(y_n)\rvert\to0$, $u(\xi)$ is well-defined and
$$|u(\xi)-u(x_n)|\le\sum_{i=n}^\infty|u(x_i)-u(x_{i+1})|\le\sum_{i=n}^\infty\sqrt{2C}(\sqrt{3\lambda})^n=\frac{\sqrt{2C}}{1-\sqrt{3\lambda}}(\sqrt{3\lambda})^n.$$

Next, we show that the extended function $u$ is continuous on $\hat{X}$. We only need to show that for all sequence $\myset{\xi_n}\subseteq\dd X$ with $\xi_n\to\xi\in\dd X$ in $\rho_a$, we have $u(\xi_n)\to u(\xi)$. Since $\dd X$ with $\rho_a$ is H\"older equivalent to $K$ with Euclidean metric by Equation (\ref{eqn_Holder}), we use them interchangeably, $\myset{\xi_n}\subseteq K$ and $\xi_n\to\xi\in K$ in Euclidean metric.

For all $\veps>0$, there exists $M\ge1$ such that $(\sqrt{3\lambda})^M<\veps$. Take $w\in W_M$ such that $\xi\in K_w$, there are at most 12 numbers of $\tilde{w}\in W_M$, $\tilde{w}\ne w$ such that $\tilde{w}\sim w$, see Figure \ref{figure_nbhd}. Indeed, if we analyze geometric property of SG carefully, we will see there are at most 3.




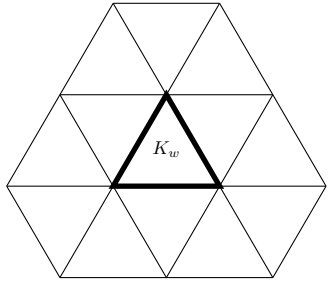
\begin{figure}[ht]
  \centering
  \begin{tikzpicture}[scale=0.7]
  \tikzstyle{every node}=[font=\small,scale=0.7]
  \draw (0,2*1.7320508076)--(2,2*1.7320508076);
  \draw (-1,1.7320508076)--(3,1.7320508076);
  \draw (-2,0)--(4,0);
  \draw (-1,-1.7320508076)--(3,-1.7320508076);
  
  \draw (-2,0)--(0,2*1.7320508076);
  \draw (-1,-1.7320508076)--(2,2*1.7320508076);
  \draw (1,-1.7320508076)--(3,1.7320508076);
  \draw (3,-1.7320508076)--(4,0);
  
  \draw (-2,0)--(-1,-1.7320508076);
  \draw (-1,1.7320508076)--(1,-1.7320508076);
  \draw (0,2*1.7320508076)--(3,-1.7320508076);
  \draw (2,2*1.7320508076)--(4,0);
  
  \draw[line width=2pt] (0,0)--(1,1.7320508076)--(2,0)--cycle;
  \draw (1,0.7) node {$K_w$};
  
  \end{tikzpicture}
  \caption{A neighborhood of $K_w$}\label{figure_nbhd}
\end{figure}



Let $U=\cup_{\tilde{w}:\tilde{w}\in W_M,\tilde{w}\sim w}K_{\tilde{w}}\cup K_w$, there exists $N\ge1$ for all $n>N$, $\xi_n\in U$. For all $n>N$. If $\xi_n\in K_w$, then
$$\lvert u(\xi_n)-u(\xi)\rvert\le\lvert u(\xi_n)-u(w)\rvert+\lvert u(\xi)-u(w)\rvert\le\frac{2\sqrt{2C}}{1-\sqrt{3\lambda}}(\sqrt{3\lambda})^M<\frac{2\sqrt{2C}}{1-\sqrt{3\lambda}}\veps.$$
If $\xi_n\in K_{\tilde{w}}$, $\tilde{w}\in K_M$, $\tilde{w}\sim w$, then
$$
\begin{aligned}
\lvert u(\xi_n)-u(\xi)\rvert&\le\lvert u(\xi_n)-u(\tilde{w})\rvert+\lvert u(\tilde{w})-u(w)\rvert+\lvert u(w)-u(\xi)\rvert\\
&\le\frac{2\sqrt{2C}}{1-\sqrt{3\lambda}}(\sqrt{3\lambda})^M+\sqrt{\frac{2C}{\min{\myset{C_1,C_2}}}}(\sqrt{3\lambda})^M\\
&<\left(\frac{2\sqrt{2C}}{1-\sqrt{3\lambda}}+\sqrt{\frac{2C}{\min{\myset{C_1,C_2}}}}\right)\veps.
\end{aligned}
$$
Hence
$$|u(\xi_n)-u(\xi)|<\left(\frac{2\sqrt{2C}}{1-\sqrt{3\lambda}}+\sqrt{\frac{2C}{\min{\myset{C_1,C_2}}}}\right)\veps,$$
for all $n>N$. $\lim_{n\to+\infty}u(\xi_n)=u(\xi)$. The extended function $u$ is continuous on $\hat{X}$.
\end{proof}

\begin{proof}[Proof of Theorem \ref{thm_main}]
Since $C_0(X)\subseteq\Ff_{\mybar{X}}$ is dense in $L^2({\mybar{X}};m)$, we have $\Ee_{\mybar{X}}$ is a symmetric form on $L^2({\mybar{X}};m)$.

We show closed property of $\Ee_{\mybar{X}}$. Let $\myset{u_k}\subseteq\Ff_{\mybar{X}}$ be an $(\Ee_{\mybar{X}})_1$-Cauchy sequence. Then there exists $u\in L^2(\mybar{X};m)$ such that $u_k\to u$ in $L^2(\mybar{X};m)$, hence $u_k(x)\to u(x)$ for all $x\in X$. By Fatou's lemma, we have
$$
\begin{aligned}
\frac{1}{2}&\sum_{x,y\in X}c(x,y)\left((u_k-u)(x)-(u_k-u)(y)\right)^2\\
&=\frac{1}{2}\sum_{x,y\in X}c(x,y)\lim_{l\to+\infty}\left((u_k-u_l)(x)-(u_k-u_l)(y)\right)^2\\
&\le\varliminf_{l\to+\infty}\frac{1}{2}\sum_{x,y\in X}c(x,y)\left((u_k-u_l)(x)-(u_k-u_l)(y)\right)^2\\
&=\varliminf_{l\to+\infty}\Ee_{\mybar{X}}(u_k-u_l,u_k-u_l).
\end{aligned}
$$
Letting $k\to+\infty$, we have
$$\frac{1}{2}\sum_{x,y\in X}c(x,y)\left((u_k-u)(x)-(u_k-u)(y)\right)^2\to0,$$
and
$$\frac{1}{2}\sum_{x,y\in X}c(x,y)\left(u(x)-u(y)\right)^2<+\infty.$$
By Lemma \ref{lem_ext}, $u$ can be extended continuously to $\mybar{X}$, hence $u\in C(\mybar{X})$, $u\in\Ff_{\mybar{X}}$. $\Ee_{\mybar{X}}$ is closed.

It is obvious that $\Ee_{\mybar{X}}$ is Markovian. Hence $\Ee_{\mybar{X}}$ is a Dirichlet form on $L^2(\mybar{X};m)$.

Since ${\mybar{X}}$ is compact, we have $C_0({\mybar{X}})=C({\mybar{X}})$. To show $\Ee_{\mybar{X}}$ is regular, we need to show $C_0({\mybar{X}})\cap\Ff_{\mybar{X}}=C({\mybar{X}})\cap\Ff_{\mybar{X}}=\Ff_{\mybar{X}}$ is $(\Ee_{\mybar{X}})_1$-dense in $\Ff_{\mybar{X}}$ and uniformly dense in $C_0({\mybar{X}})=C({\mybar{X}})$. $\Ff_{\mybar{X}}$ is trivially $(\Ee_{\mybar{X}})_1$-dense in $\Ff_{\mybar{X}}$. We need to show that $\Ff_{\mybar{X}}$ is uniformly dense in $C({\mybar{X}})$. Since $\mybar{X}$ is compact, we have $\Ff_{\mybar{X}}$ is a sub algebra of $C(\mybar{X})$. By Stone-Weierstrass theorem, we only need to show that $\Ff_{\mybar{X}}$ separates points. The idea of our proof is from classical construction of local regular Dirichlet form on SG.

For all $p,q\in\mybar{X}$ with $p\ne q$, we only need to show that there exists $v\in\Ff_{\mybar{X}}$ such that $v(p)\ne v(q)$. If $p\in X$, then let $v(p)=1$ and $v(x)=0$ for all $x\in X\backslash\myset{p}$, then
$$\sum_{x,y\in X}c(x,y)(v(x)-v(y))^2<+\infty.$$
By Lemma \ref{lem_ext}, $v$ can be extended to a function in $C(\mybar{X})$, still denoted by $v$, hence $v\in\Ff_{\mybar{X}}$. Moreover, $v(q)=0\ne1=v(p)$. If $q\in X$, then we have the proof similar to the above.

If $p,q\in\mybar{X}\backslash X=\dd X=K$, then there exists sufficiently large $m\ge1$ and $w^{(1)},w^{(2)}\in S_m$ with $p\in K_{w^{(1)}}$, $q\in K_{w^{(2)}}$ and $K_{w^{(1)}}\cap K_{w^{(2)}}=\emptyset$, hence $w^{(1)}\not\sim w^{(2)}$. Let $v=0$ in $B_m$ and
$$v(w^{(1)}0)=v(w^{(1)}1)=v(w^{(1)}2)=1.$$
For all $w\in S_{m+1}\backslash\myset{w^{(1)}0,w^{(1)}1,w^{(1)}2}$, let
$$
v(w)=
\begin{cases}
1,&\text{if }w\sim w^{(1)}0\text{ or }w\sim w^{(1)}1\text{ or }w\sim w^{(1)}2,\\
0,&\text{otherwise},
\end{cases}
$$
then
$$v(w^{(2)}0)=v(w^{(2)}1)=v(w^{(2)}2)=0.$$
In the summation $\sum_{x,y\in S_{m+1}}c(x,y)(v(x)-v(y))^2$, horizontal edges of type \Rmnum{2} make no contribution since $v$ takes same values at end nodes of each such edge. Assume that we have constructed $v$ on $B_n$ such that in the summation $\sum_{x,y\in S_i}c(x,y)(v(x)-v(y))^2$, $i=m+1,\ldots,n$, horizontal edges of type \Rmnum{2} make no contribution, that is, $v$ takes same values at end nodes of each edge. We construct $v$ on $S_{n+1}$ as follows.

Consider $\sum_{x,y\in S_n}c(x,y)(v(x)-v(y))^2$, nonzero terms all come from edges of smallest triangles in $S_n$. Pick up one such triangle in $S_n$, it generates three triangles in $S_{n+1}$, nine triangles in $S_{n+2}$, \ldots. See Figure \ref{figure_gene}.




\begin{figure}[ht]
  \centering
  \begin{tikzpicture}
  \draw (0,0)--(3,0);
  \draw (3,0)--(3/2,3/2*1.7320508076);
  \draw (0,0)--(3/2,3/2*1.7320508076);
  
  \draw (6,0)--(9,0);
  \draw (9,0)--(7.5,3/2*1.7320508076);
  \draw (6,0)--(7.5,3/2*1.7320508076);
  
  \draw (6.5,1/2*1.7320508076)--(7,0);
  \draw (8.5,1/2*1.7320508076)--(8,0);
  \draw (7,1.7320508076)--(8,1.7320508076);
  
  \draw[fill=black] (0,0) circle (0.06);
  \draw[fill=black] (3,0) circle (0.06);
  \draw[fill=black] (3/2,3/2*1.7320508076) circle (0.06);
  \draw[fill=black] (6,0) circle (0.06);
  \draw[fill=black] (7,0) circle (0.06);
  \draw[fill=black] (8,0) circle (0.06);
  \draw[fill=black] (9,0) circle (0.06);
  \draw[fill=black] (6.5,1.7320508076/2) circle (0.06);
  \draw[fill=black] (8.5,1.7320508076/2) circle (0.06);
  \draw[fill=black] (7,1.7320508076) circle (0.06);
  \draw[fill=black] (8,1.7320508076) circle (0.06);
  \draw[fill=black] (7.5,3/2*1.7320508076) circle (0.06);
  
  \draw (0,-0.3) node {$a$};
  \draw (3,-0.3) node {$b$};
  \draw (1.5,3/2*1.7320508076+0.3) node {$c$};
  
  \draw (6,-0.3) node {$a$};
  \draw (9,-0.3) node {$b$};
  \draw (7.5,3/2*1.7320508076+0.3) node {$c$};
  
  \draw (7,-0.3) node {$x$};
  \draw (8,-0.3) node {$x$};
  
  \draw (6.3,1/2*1.7320508076) node {$z$};
  \draw (6.8,1.7320508076) node {$z$};
  
  \draw (8.7,1/2*1.7320508076) node {$y$};
  \draw (8.2,1.7320508076) node {$y$};
  
  \end{tikzpicture}
  \caption{Generation of triangles}\label{figure_gene}
\end{figure}
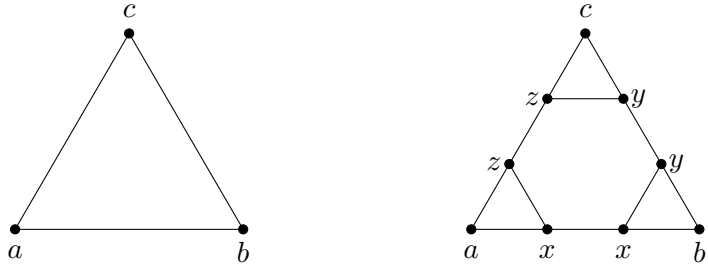



We only need to assign values of $v$ on the three triangles in $S_{n+1}$ from the values of $v$ on the triangle in $S_n$. As in Figure \ref{figure_gene}, $x,y,z$ are values of $v$ at corresponding nodes to be determined from $a,b,c$. The contribution of this one triangle in $S_n$ to $\sum_{x,y\in S_n}c(x,y)(v(x)-v(y))^2$ is
$$A_1=\frac{C_1}{(3\lambda)^n}\left[(a-b)^2+(b-c)^2+(a-c)^2\right].$$
The contribution of these three triangles in $S_{n+1}$ to $\sum_{x,y\in S_{n+1}}c(x,y)(v(x)-v(y))^2$ is
$$
\begin{aligned}
A_2=\frac{C_1}{(3\lambda)^{n+1}}&\left[(a-x)^2+(a-z)^2+(x-z)^2\right.\\
&+(b-x)^2+(b-y)^2+(x-y)^2\\
&\left.+(c-y)^2+(c-z)^2+(y-z)^2\right].\\
\end{aligned}
$$
Consider $A_2$ as a function of $x,y,z$, by elementary calculation, $A_2$ takes minimum value when
$$
\begin{cases}
x=\frac{2a+2b+c}{5},\\
y=\frac{a+2b+2c}{5},\\
z=\frac{2a+b+2c}{5},
\end{cases}
$$
and
$$
\begin{aligned}
A_2&=\frac{C_1}{(3\lambda)^{n+1}}\cdot\frac{3}{5}\left[(a-b)^2+(b-c)^2+(a-c)^2\right]\\
&=\frac{1}{5\lambda}\left(\frac{C_1}{(3\lambda)^n}\left[(a-b)^2+(b-c)^2+(a-c)^2\right]\right)=\frac{1}{5\lambda}A_1.
\end{aligned}
$$
By construction, horizontal edges of type \Rmnum{2} in $S_{n+1}$ make no contribution to
$$\sum_{x,y\in S_{n+1}}c(x,y)(v(x)-v(y))^2$$
and
$$\sum_{x,y\in S_{n+1}}c(x,y)(v(x)-v(y))^2=\frac{1}{5\lambda}\sum_{x,y\in S_{n}}c(x,y)(v(x)-v(y))^2.$$
Since $\lambda>1/5$, we have
$$\sum_{n=0}^\infty\sum_{x,y\in S_{n}}c(x,y)(v(x)-v(y))^2<+\infty,$$
this is the contribution of all horizontal edges to $\sum_{x,y\in X}c(x,y)(v(x)-v(y))^2$. We consider the contribution of all vertical edges. For all $n\ge m$, by construction $v|_{S_{n+1}}$ is uniquely determined by $v|_{S_n}$, hence the contribution of vertical edges between $S_n$ and $S_{n+1}$ is uniquely determined by $v|_{S_n}$. As above, we pick one smallest triangle in $S_n$ and consider the contribution of the vertical edges connecting it to $S_{n+1}$. There are nine vertical edges between $S_n$ and $S_{n+1}$ connecting this triangle. These nine vertical edges make contribution
$$
\begin{aligned}
A_3&=\frac{1}{(3\lambda)^n}\left[(a-x)^2+(a-z)^2+(a-a)^2\right.\\
&+(b-x)^2+(b-y)^2+(b-b)^2\\
&+(c-y)^2+(c-z)^2+(c-c)^2\left.\right]\\
&=\frac{14}{25C_1}\left(\frac{C_1}{(3\lambda)^n}\left[(a-b)^2+(b-c)^2+(a-c)^2\right]\right)=\frac{14}{25C_1}A_1.
\end{aligned}
$$
Hence
$$\sum_{x\in S_n,y\in S_{n+1}}c(x,y)(v(x)-v(y))^2=\frac{14}{25C_1}\sum_{x,y\in S_n}c(x,y)(v(x)-v(y))^2,$$
and
$$\sum_{n=0}^\infty\sum_{x\in S_n,y\in S_{n+1}}c(x,y)(v(x)-v(y))^2<+\infty\Leftrightarrow\sum_{n=0}^\infty\sum_{x,y\in S_n}c(x,y)(v(x)-v(y))^2<+\infty.$$
Since $\lambda>1/5$, we have both summations converge and
$$\sum_{x,y\in X}c(x,y)(v(x)-v(y))^2<+\infty.$$
By Lemma \ref{lem_ext}, $v$ can be extended to a function in $C(\mybar{X})$, still denoted by $v$, hence $v\in\Ff_{\mybar{X}}$. Since $v$ is constructed by convex interpolation on $X\backslash B_{m+1}$, we have $v(p)=1\ne0=v(q)$.

Therefore, $(\Ee_{\mybar{X}},\Ff_{\mybar{X}})$ is a regular Dirichlet form on $L^2({\mybar{X}};m)$.
\end{proof}

\begin{mythm}
$\Ee_{\mybar{X}}$ on $L^2({\mybar{X}};m)$ is a regular representation of $\Ee^\rref$ on $L^2(X;m)$.
\end{mythm}

Regular representation theory was developed by Fukushima \cite{Fuk71} and incorporated into his book \cite[Appendix, A4]{FOT11}.
\begin{proof}
We only need to construct an algebraic isomorphism $\Phi:(\Ff^\rref_a)_b\to(\Ff_{\mybar{X}})_b$ such that for all $u\in(\Ff^\rref_a)_b$
\begin{equation}\label{eqn_iso}
\lVert u\rVert_{L^\infty(X;m)}=\lVert\Phi(u)\rVert_{L^\infty(\mybar{X};m)},(u,u)_X=(\Phi(u),\Phi(u))_{\mybar{X}}, \Ee^\rref(u,u)=\Ee_{\mybar{X}}(\Phi(u),\Phi(u)).
\end{equation}
Indeed, for all $u\in(\Ff^\rref_a)_b$, we have $\sum_{x,y\in X}c(x,y)(u(x)-u(y))^2<+\infty$, by Lemma \ref{lem_ext}, define $\Phi(u)$ as the continuous extension of $u$ to ${\mybar{X}}$. Since $\Ee^\rref$, $\Ee_{\mybar{X}}$ have the same expression for energy and $m(\dd X)=0$, Equation (\ref{eqn_iso}) is obvious.
\end{proof}

Moreover we have

\begin{mythm}\label{thm_part}
$(\Ee_X,\Ff_X)$ on $L^2(X;m)$ is part form on $X$ of $(\Ee_{\mybar{X}},\Ff_{\mybar{X}})$ on $L^2({\mybar{X}};m)$.
\end{mythm}

\begin{proof}
By \cite[Theorem 3.3.9]{CF12}, since $X\subseteq{\mybar{X}}$ is an open subset and $\Ff_{\mybar{X}}$ is a special standard core of $(\Ee_{\mybar{X}},\Ff_{\mybar{X}})$ on $L^2(\mybar{X};m)$, we have
$$(\Ff_{\mybar{X}})_X=\myset{u\in\Ff_{\mybar{X}}:\mathrm{supp}(u)\subseteq X}=\myset{u\in\Ff_{\mybar{X}}:u\in C_0(X)}=C_0(X).$$
Since $\Ff_X$ is the $(\Ee_X)_1$-closure of $C_0(X)$, we have part form of $(\Ee_{\mybar{X}},\Ff_{\mybar{X}})$ on $L^2({\mybar{X}};m)$ on $X$ is exactly $(\Ee_X,\Ff_X)$ on $L^2(X;m)$.
\end{proof}

From probabilistic viewpoint, $(\Ee_X,\Ff_X)$ on $L^2(X;m)$ corresponds to absorbed process $\myset{X_t}$ and $(\Ee_{\mybar{X}},\Ff_{\mybar{X}})$ on $L^2({\mybar{X}};m)$ corresponds to reflected process $\myset{\mybar{X}_t}$. By \cite[Theorem 3.3.8]{CF12}, $\myset{X_t}$ is part process of $\myset{\mybar{X}_t}$ on $X$ which can be described as follows.

Let
$$\tau_X=\inf{\myset{t>0:\mybar{X}_t\notin X}}=\inf{\myset{t>0:\mybar{X}_t\in\dd X}}=\sigma_{\dd X},$$
then
$$X_t=
\begin{cases}
\mybar{X}_t,&0\le t<\tau_X,\\
\dd,&t\ge\tau_X,
\end{cases}
$$
and
$$\zeta=\tau_X=\sigma_{\dd X}.$$

\section{Trace Form on $\dd X$}\label{sec_trace}

In this section, we take trace of the regular Dirichlet form $(\Ee_{\mybar{X}},\Ff_{\mybar{X}})$ on $L^2(\mybar{X};m)$ to $K$ to have a regular Dirichlet form $(\Ee_K,\Ff_K)$ on $L^2(K;\nu)$ with the form (\ref{eqn_E_K}).

First, we show that $\nu$ is of finite energy with respect to $\Ee_{\mybar{X}}$, that is, 
\begin{equation*}
\int_{\mybar{X}}\lvert u(x)\rvert\nu(\md x)\le C\sqrt{(\Ee_{\mybar{X}})_1(u,u)}\text{ for all }u\in\Ff_{\mybar{X}}\cap C_0({\mybar{X}})=\Ff_{\mybar{X}},
\end{equation*}
where $C$ is some positive constant. Since $\nu(\dd X)=1$, we only need to show that
\begin{mythm}\label{thm_trace}
\begin{equation}\label{eqn_trace}
\left(\int_{\mybar{X}}\lvert u(x)\rvert^2\nu(\md x)\right)^{1/2}\le C\sqrt{(\Ee_{\mybar{X}})_1(u,u)}\text{ for all }u\in\Ff_{\mybar{X}}.
\end{equation}
\end{mythm}

We need some preparation.
\begin{mythm}(\cite[Theorem 1.1]{ALP99})\label{thm_ALP}
Suppose that a reversible random walk $\myset{Z_n}$ is transient, then for all $f$ with
$$D(f)=\frac{1}{2}\sum_{x,y\in X}c(x,y)(f(x)-f(y))^2<+\infty,$$
we have $\myset{f(Z_n)}$ converges almost surely and in $L^2$ under $\PP_x$ for all $x\in X$.
\end{mythm}

For all $f$ with $D(f)<+\infty$, under $\PP_o$, $\myset{f(Z_n)}$ converges almost surely and in $L^2$ to a random variable $W$, that is
$$f(Z_n)\to W,\PP_o\text{-a.s.},\EE_o\left[\left(f(Z_n)-W\right)^2\right]\to0,$$
then $W$ is a terminal random variable. By Theorem \ref{thm_conv}, $Z_n\to Z_\infty$, $\PP_o$-a.s.. By \cite[Corollary 7.65]{Wo09}, $W$ is of the form $W=\vphi(Z_\infty)$, $\PP_o$-a.s., where $\vphi$ is a measurable function on $\dd X$, we define a map $f\mapsto\vphi$, this is the operation of taking boundary value in some sense.

Let $\DD=\myset{f:D(f)<+\infty}$. The Dirichlet norm of $f\in\DD$ is given by $\lVert f\rVert^2=D(f)+\pi(o)f(o)^2$. Let $\DD_0$ be the family of all functions that are limits in the Dirichlet norm of functions with finite support. We have the following Royden decomposition.

\begin{mythm}(\cite[Theorem 3.69]{Soa94})\label{thm_Soa}
For all $f\in\DD$, there exist unique harmonic Dirichlet function $f_{HD}$ and $f_0\in\DD_0$ such that $f=f_{HD}+f_0$. Moreover, $D(f)=D(f_{HD})+D(f_0)$.
\end{mythm}

\begin{mylem}(\cite[Lemma 2.1]{ALP99})\label{lem_ALP}
For all $f\in\DD_0$, $x\in X$, we have
$$\pi(x)f(x)^2\le D(f)G(x,x).$$
Furthermore, there exists a superharmonic function $h\in\DD_0$ such that $h\ge|f|$ point wise and $D(h)\le D(f)$.
\end{mylem}

\begin{proof}[Proof of Theorem \ref{thm_trace}]
Since $\Ff_{\mybar{X}}\subseteq C(\mybar{X})$, for all $u\in\Ff_{\mybar{X}}$, it is trivial to take boundary value just as $u|_{\dd X}$. We still use notions $f,\vphi$. We have
$$f(Z_n)\to\vphi(Z_\infty),\PP_o\text{-a.s.},\EE_o\left[\left(f(Z_n)-\vphi(Z_\infty)\right)^2\right]\to0.$$
Under $\PP_o$, the hitting distribution of $\myset{Z_n}$ or the distribution of $Z_\infty$ is $\nu$, normalized Hausdorff measure on $K$, we have
$$\int_{\dd X}|\vphi|^2\md\nu=\EE_o\left[\vphi(Z_\infty)^2\right]=\lim_{n\to+\infty}\EE_o\left[f(Z_n)^2\right].$$
We only need to estimate $\EE_o\left[f(Z_n)^2\right]$ in terms of
$$D(f)+(f,f)=\frac{1}{2}\sum_{x,y\in X}c(x,y)(f(x)-f(y))^2+\sum_{x\in X}f(x)^2m(x).$$
By Theorem \ref{thm_Soa}, we only need to consider harmonic Dirichlet functions and functions in $\DD_0$.

For all $f\in\DD$, we have
$$
\begin{aligned}
&\sum_{k=0}^\infty\EE_o\left[\left(f(Z_{k+1})-f(Z_k)\right)^2\right]=\sum_{k=0}^\infty\EE_o\left[\EE_o\left[\left(f(Z_{k+1})-f(Z_k)\right)^2|Z_k\right]\right]\\
&=\sum_{k=0}^\infty\EE_o\left[\EE_{Z_k}\left[\left(f(Z_{1})-f(Z_0)\right)^2\right]\right]=\sum_{k=0}^\infty\sum_{x\in X}P^{(k)}(o,x)\EE_{x}\left[\left(f(Z_{1})-f(Z_0)\right)^2\right]\\
&=\sum_{x,y\in X}\left(\sum_{k=0}^\infty P^{(k)}(o,x)\right)P(x,y)\left(f(x)-f(y)\right)^2=\sum_{x,y\in X}G(o,x)\frac{c(x,y)}{\pi(x)}(f(x)-f(y))^2\\
&=\sum_{x,y\in X}\frac{\pi(x)G(x,o)}{\pi(o)}\frac{c(x,y)}{\pi(x)}(f(x)-f(y))^2=\sum_{x,y\in X}\frac{F(x,o)G(o,o)}{\pi(o)}c(x,y)(f(x)-f(y))^2\\
&\le\frac{G(o,o)}{\pi(o)}\sum_{x,y\in X}c(x,y)(f(x)-f(y))^2=\frac{2G(o,o)}{\pi(o)}D(f).
\end{aligned}
$$

Let $f$ be a harmonic Dirichlet function, then $\myset{f(Z_n)}$ is a martingale. For all $n\ge1$
$$
\begin{aligned}
\EE_o\left[f(Z_n)^2\right]&=\EE_o\left[\left(\sum_{k=0}^{n-1}(f(Z_{k+1})-f(Z_k))+f(Z_0)\right)^2\right]\\
&=\sum_{k=0}^{n-1}\EE_o\left[(f(Z_{k+1})-f(Z_k))^2\right]+f(o)^2\\
&\le\sum_{k=0}^{\infty}\EE_o\left[(f(Z_{k+1})-f(Z_k))^2\right]+f(o)^2\\
&\le\frac{2G(o,o)}{\pi(o)}D(f)+f(o)^2,
\end{aligned}
$$
hence
\begin{equation}\label{eqn_hd}
\EE_o\left[f_{HD}(Z_n)^2\right]\le\frac{2G(o,o)}{\pi(o)}D(f_{HD})+f_{HD}(o)^2.
\end{equation}

Let $f\in\DD_0$. Let $h$ be as in Lemma \ref{lem_ALP}. Then $h\ge0$. Since $h$ is superharmonic, we have
$$\EE_o\left[h(Z_{k+1})-h(Z_k)|Z_0,\ldots,Z_k\right]\le0,$$
hence
$$
\begin{aligned}
&\EE_o\left[h(Z_{k+1})^2-h(Z_k)^2\right]\\
&=\EE_o\left[(h(Z_{k+1})-h(Z_{k}))^2\right]+2\EE_o\left[h(Z_k)(h(Z_{k+1})-h(Z_k))\right]\\
&=\EE_o\left[(h(Z_{k+1})-h(Z_{k}))^2\right]+2\EE_o\left[\EE_o\left[h(Z_k)(h(Z_{k+1})-h(Z_k))|Z_0,\ldots,Z_k\right]\right]\\
&=\EE_o\left[(h(Z_{k+1})-h(Z_{k}))^2\right]+2\EE_o\left[h(Z_k)\EE_o\left[h(Z_{k+1})-h(Z_k)|Z_0,\ldots,Z_k\right]\right]\\
&\le\EE_o\left[(h(Z_{k+1})-h(Z_{k}))^2\right].
\end{aligned}
$$
We have
$$
\begin{aligned}
\EE_o\left[h(Z_n)^2\right]&=\sum_{k=0}^{n-1}\EE_o\left[h(Z_{k+1})^2-h(Z_k)^2\right]+h(o)^2\\
&\le\sum_{k=0}^{n-1}\EE_o\left[(h(Z_{k+1})-h(Z_{k}))^2\right]+\frac{G(o,o)}{\pi(o)}D(h)\\
&\le\sum_{k=0}^{\infty}\EE_o\left[(h(Z_{k+1})-h(Z_{k}))^2\right]+\frac{G(o,o)}{\pi(o)}D(h)\\
&\le\frac{2G(o,o)}{\pi(o)}D(h)+\frac{G(o,o)}{\pi(o)}D(h)\\
&=\frac{3G(o,o)}{\pi(o)}D(h),
\end{aligned}
$$
hence
$$\EE_o\left[f(Z_n)^2\right]\le\EE_o\left[h(Z_n)^2\right]\le\frac{3G(o,o)}{\pi(o)}D(h)\le\frac{3G(o,o)}{\pi(o)}D(f).$$
We have
\begin{equation}\label{eqn_d0}
\EE_o\left[f_0(Z_n)^2\right]\le\frac{3G(o,o)}{\pi(o)}D(f_0).
\end{equation}
Combining Equation (\ref{eqn_hd}) and Equation (\ref{eqn_d0}), we have
$$
\begin{aligned}
\EE_o\left[f(Z_n)^2\right]&=\EE_o\left[(f_{HD}(Z_n)+f_0(Z_n))^2\right]\le2\EE_o\left[f_{HD}(Z_n)^2+f_0(Z_n)^2\right]\\
&\le2\left(\frac{2G(o,o)}{\pi(o)}D(f_{HD})+f_{HD}(o)^2+\frac{3G(o,o)}{\pi(o)}D(f_0)\right)\\
&\le2\left(\frac{5G(o,o)}{\pi(o)}D(f)+(f(o)-f_0(o))^2\right)\\
&\le2\left(\frac{5G(o,o)}{\pi(o)}D(f)+2f(o)^2+2f_0(o)^2\right)\\
&\le2\left(\frac{5G(o,o)}{\pi(o)}D(f)+2\frac{1}{m(o)}f(o)^2m(o)+2\frac{G(o,o)}{\pi(o)}D(f_0)\right)\\
&\le2\left(\frac{7G(o,o)}{\pi(o)}D(f)+2\frac{1}{m(o)}\sum_{x\in X}f(x)^2m(x)\right)\\
&\le\max{\myset{\frac{14G(o,o)}{\pi(o)},\frac{4}{m(o)}}}\left(D(f)+\sum_{x\in X}f(x)^2m(x)\right).
\end{aligned}
$$
Let $C^2=\max{\myset{\frac{14G(o,o)}{\pi(o)},\frac{4}{m(o)}}}$ be a constant only depending on conductance $c$ and measure $m$, we have
$$\int_{\dd X}|\vphi|^2\md\nu=\lim_{n\to+\infty}\EE_o\left[f(Z_n)^2\right]\le C^2(D(f)+\sum_{x\in X}f(x)^2m(x)).$$
In the notion of $u$, we obtain Equation (\ref{eqn_trace}).
\end{proof}

Second, we obtain a regular Dirichlet form on $L^2(\dd X;\nu)$ by abstract theory of trace form. More detailed discussion of trace form, see \cite[Chapter 5, 5.2]{CF12} and \cite[Chapter 6, 6.2]{FOT11}. We introduce some results used here.

Taking trace with respect to a regular Dirichlet form corresponds to taking time-change with respect to corresponding Hunt process. Taking trace is realized by smooth measure. The family of all smooth measures is denoted by $S$. Taking time-change is realized by positive continuous additive functional, abbreviated as PCAF. The family of all PCAFs is denoted by $A_c^+$. The family of all equivalent classes of $A_c^+$ and the family $S$ are in one-to-one correspondence, see \cite[Theorem 5.1.4]{FOT11}.

We fix a regular Dirichlet form $(\Ee,\Ff)$ on $L^2(E;m)$ and its corresponding Hunt process $X=\myset{X_t}$.
\begin{itemize}
\item First, we introduce basic setup of time-change. Given a PCAF $A\in A_c^+$, define its support $F$, then $F$ is quasi closed and nearly Borel measurable. Define the right-continuous inverse $\tau$ of $A$, let $\check{X}_t=X_{\tau_t}$, then $\check{X}$ is a right process with state space $F$ and called the time-changed process of $X$ by $A$.
\item Second, we introduce basic setup of trace form. For arbitrary non-polar, quasi closed, nearly Borel measurable, finely closed set $F$, define hitting distribution $H_F$ of ${X}$ for $F$ as follows:
$$H_Fg(x)=\EE_x\left[g(X_{\sigma_F})\indi_{\sigma_F<+\infty}\right],x\in E,g\text{ is nonnegative Borel function.}$$
By \cite[Theorem 3.4.8]{CF12}, for all $u\in\Ff_e$, we have $H_F|u|(x)<+\infty$, q.e. and $H_Fu\in\Ff_e$. Define
$$\check{\Ff}_e=\Ff_e|_F,\check{\Ee}(u|_F,v|_F)={\Ee}(H_Fu,H_Fv),u,v\in\Ff_e.$$
Two elements in $\check{\Ff}_e$ can be identified if they coincide q.e. on $F$. We still need a measure on $F$. Let
$$S_F=\myset{\mu\in S:\text{the quasi support of }\mu=F},$$
where the quasi support of a Borel measure is the smallest (up to q.e. equivalence) quasi closed set outside which the measure vanishes. Let $\mu\in S_F$, by \cite[Theorem 3.3.5]{CF12}, two elements of $\check{\Ff}_e$ coincide q.e. on $F$ if and only if they coincide $\mu$-a.e.. Define $\check{\Ff}=\check{\Ff}_e\cap L^2(F;\mu)$. Then $(\check{\Ee},\check{\Ff})$ is a symmetric form on $L^2(F;\mu)$.
\item Third, the relation of trace form and time-change process is as follows. Given $A\in A_c^+$ or equivalently $\mu\in S$, let $F$ be the support of $A$, then $F$ satisfies the conditions in the second setup and by \cite[Theorem 5.2.1(\rmnum{1})]{CF12}, $\mu\in S_F$. We obtain $(\check{\Ee},\check{\Ff})$ on $L^2(F;\mu)$. By \cite[Theorem 5.2.2]{CF12}, the regular Dirichlet form corresponding to $\check{X}$ is exactly $(\check{\Ee},\check{\Ff})$ on $L^2(F;\mu)$.
\end{itemize}

We have $F\subseteq\mathrm{supp}(\mu)$ q.e.. But the point is that $F$ can be strictly contained in $\mathrm{supp}(\mu)$ q.e., usually we indeed need a trace form on $\mathrm{supp}(\mu)$. \cite{CF12} provided a solution not for all smooth measures, but some subset
$$\mathring{S}=\myset{\mu:\text{positive Radon measure charging no }\Ee\text{-polar set}}.$$
For non-$\Ee$-polar, quasi closed subset $F$ of $E$, let
$$\mathring{S}_F=\myset{\mu\in\mathring{S}:\text{the quasi support of }\mu\text{ is }F}.$$
Note that if $\mu\in\mathring{S}_F$, it may happen that $\mathrm{supp}(\mu)\supsetneqq F$ q.e.. We want some $\mu\in\mathring{S}_F$ such that $\mathrm{supp}(\mu)=F$ q.e.. \cite{CF12} gave a criterion as follows.

\begin{mylem}(\cite[Lemma 5.2.9(\rmnum{2})]{CF12})\label{lem_ac}
Let $F$ be a non-$\Ee$-polar, nearly Borel, finely closed set. Let $\nu\in\mathring{S}$ satisfy $\nu(E\backslash F)=0$. Assume the 1-order hitting distribution $H_F^1(x,\cdot)$ of $X$ for $F$ is absolutely continuous with respect to $\nu$ for $m$-a.e. $x\in E$. Then $\nu\in\mathring{S}_F$.
\end{mylem}

\begin{mycor}(\cite[Corollary 5.2.10]{CF12})\label{cor_cf}
Let $F$ be a closed set. If there exists $\nu\in\mathring{S}_F$ such that the topological support $\mathrm{supp}(\nu)=F$, then for all $\mu\in\mathring{S}_F$, we have $(\check{\Ee},\check{\Ff})$ is a regular Dirichlet form on $L^2(F;\mu)$.
\end{mycor}

Roughly speaking, given a positive Radon measure $\mu$ charging no $\Ee$-polar set, let $F=\mathrm{supp}(\mu)$. First check Lemma \ref{lem_ac} to have $\mu\in\mathring{S}_F$, then the quasi support of $\mu$ is $F$ and the support of corresponding PCAF $A$ can be taken as $F$. Second, by Corollary \ref{cor_cf}, the time-changed process $\check{X}$ of $X$ by $A$ corresponds to the regular Dirichlet form $(\check{\Ee},\check{\Ff})$ on $L^2(F;\mu)$.

Return to our case, $\nu$ is a probability measure of finite energy integral, hence $\nu$ is a positive Radon measure charging no $\Ee_{\mybar{X}}$-polar set. We need to check absolutely continuous condition in Lemma \ref{lem_ac}. We give a theorem as follows.
\begin{mythm}\label{thm_hit}
The hitting distributions of $\myset{\mybar{X}_t}$ and $\myset{Z_n}$ for $\dd X$ coincide.
\end{mythm}

\begin{proof}
Recall that $\myset{X_t}$ is characterized by random walk $\myset{Y_n}$ and jumping times $\myset{J_n}$, $\myset{X_t}$ is part process of $\myset{\mybar{X}_t}$ on $X$ and $\zeta=\tau_X=\sigma_{\dd X}<+\infty$, $\PP_x$-a.s. for all $x\in X$.

First, we show that jumping times $J_n$ are stopping times of $\myset{\mybar{X}_t}$ for all $n\ge0$. Let $\myset{\Ff_t}$ and $\myset{\mybar{\Ff}_t}$ be the minimum completed admissible filtration with respect to $\myset{X_t}$ and $\myset{\mybar{X}_t}$, respectively. By Proposition \ref{prop_jump}, $J_n$ are stopping times of $\myset{X_t}$. Since for all Borel measurable set $B\subseteq\mybar{X}$, we have
$$\myset{X_t\in B}=\myset{\mybar{X}_t\in B\cap X}\cap\myset{t<\zeta}\in\mybar{\Ff}_t,$$
$\Ff_t\subseteq\mybar{\Ff}_t$ for all $t\ge0$. $J_n$ are stopping times of $\myset{\mybar{X}_t}$ for all $n\ge0$.

Then, since $J_n\uparrow\zeta=\sigma_{\dd X}$, by quasi left continuity of $\myset{\mybar{X}_t}$, we have for all $x\in X$
$$\PP_x\left[\lim_{n\to+\infty}\mybar{X}_{J_n}=\mybar{X}_{\sigma_{\dd X}},\sigma_{\dd X}<+\infty\right]=\PP_x\left[\sigma_{\dd X}<+\infty\right],$$
that is,
$$\PP_x[\lim_{n\to+\infty}\mybar{X}_{J_n}=\mybar{X}_{\sigma_{\dd X}}]=1.$$
Note that $J_n<\zeta=\sigma_{\dd X}$, we have $\mybar{X}_{J_n}=X_{J_n}=Y_n$, hence
$$\PP_x\left[\lim_{n\to+\infty}Y_n=\mybar{X}_{\sigma_{\dd X}}\right]=1.$$
Hence the hitting distributions of $\myset{\mybar{X}_t}$ and $\myset{Z_n}$ for $\dd X$ coincide under $\PP_x$ for all $x\in X$.
\end{proof}
By Theorem \ref{thm_hit}, the hitting distribution of $\myset{\mybar{X}_t}$ for $\dd X$ under $\PP_x$ is exactly $\nu_x$, hence
\begin{equation}\label{eqn_hit}
\begin{aligned}
H_{\dd X}g(x)&=\EE_x\left[g(\mybar{X}_{\sigma_{\dd X}})\indi_{\myset{\sigma_{\dd X}<+\infty}}\right]=\EE_x\left[g(\mybar{X}_{\sigma_{\dd X}})\right]\\
&=\int_{\dd X}g\md\nu_x=\int_{\dd X}K(x,\cdot)g\md\nu=Hg(x),
\end{aligned}
\end{equation}
for all $x\in X$ and nonnegative Borel function $g$.

By Theorem \ref{thm_hit} and Theorem \ref{thm_conv}, $\nu$ satisfies the condition of Lemma \ref{lem_ac} with $F=\mathrm{supp}(\nu)=\dd X$. By above remark, we obtain a regular Dirichlet form $\check{\Ee}$ on $L^2(\dd X;\nu)$.

Third, we obtain explicit representation of $\check{\Ee}$ as follows.

\begin{mythm}\label{thm_check_E}
We have
$$
\begin{cases}
\check{\Ee}(u,u)\asymp\int_K\int_K\frac{(u(x)-u(y))^2}{|x-y|^{\alpha+\beta}}\nu(\md x)\nu(\md y)<+\infty,\\
\check{\Ff}=\myset{u\in C(K):\int_K\int_K\frac{(u(x)-u(y))^2}{|x-y|^{\alpha+\beta}}\nu(\md x)\nu(\md y)<+\infty},
\end{cases}
$$
where $\beta\in(\alpha,\beta^*)$.
\end{mythm}

To prove this theorem, we need some preparation.

\begin{mylem}\label{lem_bdy_har}
If $\lambda<1/3$, then for all $u\in C(\dd X)=C(K)$ with
$$\int_K\int_K\frac{(u(x)-u(y))^2}{|x-y|^{\alpha+\beta}}\nu(\md x)\nu(\md y)<+\infty,$$
let $v\in C(\mybar{X})$ be the extended function of $Hu$ in Lemma \ref{lem_ext}, we have $v|_{\dd X}=u$.
\end{mylem}

We need a calculation result from \cite[Theorem 5.3]{KLW17} as follows.
\begin{equation}\label{eqn_Martin}
K(x,\xi)\asymp\lambda^{|x|-|x\wedge\xi|}(\frac{1}{2})^{-\frac{\log3}{\log2}|x\wedge\xi|}=\lambda^{|x|}\left(\frac{3}{\lambda}\right)^{|x\wedge\xi|},
\end{equation}
where $x\in X$ and $\xi\in\dd X$.

\begin{proof}
By estimate of Na\"im kernel, we have
$$\sum_{x,y\in X}c(x,y)(Hu(x)-Hu(y))^2<+\infty,$$
hence Lemma \ref{lem_ext} can be applied here and $v$ is well-defined. We only need to show that for all $\myset{x_n}\subseteq X$ and $\xi\in\dd X$ with $x_n\to\xi$, then $Hu(x_n)\to u(\xi)$ as $n\to+\infty$. Indeed, since
$$Hu(x)=\int_{\dd X}K(x,\eta)u(\eta)\nu(\md\eta)=\int_{\dd X}u(\eta)\nu_x(\md\eta)=\EE_x\left[u(Z_\infty)\right],$$
we have
$$H1(x)=\int_{\dd X}K(x,\eta)\nu(\md\eta)=1$$
for all $x\in X$. Then
$$
\begin{aligned}
\lvert Hu(x_n)-u(\xi)\rvert&=\lvert\int_{\dd X}K(x_n,\eta)u(\eta)\nu(\md\eta)-u(\xi)\rvert=\lvert\int_{\dd X}K(x,\eta)(u(\eta)-u(\xi))\nu(\md\eta)\rvert\\
&\le\int_{\dd X}K(x_n,\eta)|u(\eta)-u(\xi)|\nu(\md\eta).
\end{aligned}
$$
Since $u\in C(\dd X)$, for all $\veps>0$, there exists $\delta>0$ such that for all $\eta,\xi\in\dd X$ with $\theta_a(\eta,\xi)<\delta$, we have $|u(\eta)-u(\xi)|<\veps$. Assume that $|u(x)|\le M<+\infty$ for all $x\in\dd X$, then
$$
\begin{aligned}
&\lvert Hu(x_n)-u(\xi)\rvert\\
&\le\int_{\theta_a(\eta,\xi)<\delta}K(x_n,\eta)|u(\eta)-u(\xi)|\nu(\md\eta)+\int_{\theta_a(\eta,\xi)\ge\delta}K(x_n,\eta)|u(\eta)-u(\xi)|\nu(\md\eta)\\
&<\veps\int_{\theta_a(\eta,\xi)<\delta}K(x_n,\eta)\nu(\md\eta)+2M\int_{\theta_a(\eta,\xi)\ge\delta}K(x_n,\eta)\nu(\md\eta)\\
&\le\veps+2M\int_{\theta_a(\eta,\xi)\ge\delta}K(x_n,\eta)\nu(\md\eta).
\end{aligned}
$$
There exists $N\ge1$ such that for all $n>N$, we have $\theta_a(x_n,\xi)<\delta/2$, then for all $\theta_a(\eta,\xi)\ge\delta$
$$\theta_a(x_n,\eta)\ge\theta_a(\eta,\xi)-\theta_a(x_n,\xi)\ge\delta-\frac{\delta}{2}=\frac{\delta}{2}.$$
By Equation (\ref{eqn_Martin}), we have
$$
\begin{aligned}
K(x_n,\eta)&\asymp \lambda^{|x_n|}\left(\frac{3}{\lambda}\right)^{|x_n\wedge\eta|}=\lambda^{|x_n|}e^{|x_n\wedge\eta|\log(\frac{3}{\lambda})}=\lambda^{|x_n|}e^{-a|x_n\wedge\eta|\frac{1}{a}\log(\frac{\lambda}{3})}\\
&=\lambda^{|x_n|}\rho_a(x_n,\eta)^{\frac{1}{a}\log(\frac{\lambda}{3})}=\frac{\lambda^{|x_n|}}{\rho_a(x_n,\eta)^{\frac{1}{a}\log(\frac{3}{\lambda})}}.
\end{aligned}
$$
Since $\rho_a$ and $\theta_a$ are equivalent, there exists some positive constant $C$ independent of $x_n$ and $\eta$ such that
$$K(x_n,\eta)\le C\frac{\lambda^{|x_n|}}{\delta^{\frac{1}{a}\log(\frac{3}{\lambda})}}.$$
Hence
$$\lvert Hu(x_n)-u(\xi)\rvert\le\veps+2M\int_{\theta_a(\eta,\xi)\ge\delta}C\frac{\lambda^{|x_n|}}{\delta^{\frac{1}{a}\log(\frac{3}{\lambda})}}\nu(\md\eta)\le\veps+2MC\frac{\lambda^{|x_n|}}{\delta^{\frac{1}{a}\log(\frac{3}{\lambda})}},$$
letting $n\to+\infty$, we have $|x_n|\to+\infty$, hence
$$\varlimsup_{n\to+\infty}\lvert Hu(x_n)-u(\xi)\rvert\le\veps$$
for all $\veps>0$. Since $\veps>0$ is arbitrary, we have $\lim_{n\to+\infty}Hu(x_n)=u(\xi)$.
\end{proof}

\begin{mythm}\label{thm_f_e}
$(\Ff_{\mybar{X}})_e=\Ff_{\mybar{X}}$, here we use the convention that functions in extended Dirichlet spaces are quasi continuous.
\end{mythm}

\begin{proof}
It is obvious that $(\Ff_{\mybar{X}})_e\supseteq\Ff_{\mybar{X}}$. For all $u\in(\Ff_{\mybar{X}})_e$, by definition, there exists an $\Ee_{\mybar{X}}$-Cauchy sequence $\myset{u_n}\subseteq\Ff_{\mybar{X}}$ that converges to $u$ $m$-a.e.. Hence $u_n(x)\to u(x)$ for all $x\in X$. By Fatou's lemma, we have
$$
\begin{aligned}
\frac{1}{2}\sum_{x,y\in X}c(x,y)(u(x)-u(y))^2&=\frac{1}{2}\sum_{x,y\in X}\lim_{n\to+\infty}c(x,y)(u(x)-u(y))^2\\
&\le\varliminf_{n\to+\infty}\frac{1}{2}\sum_{x,y\in X}c(x,y)(u_n(x)-u_n(y))^2\\
&=\varliminf_{n\to+\infty}\Ee_{\mybar{X}}(u_n,u_n)\\
&<+\infty.
\end{aligned}
$$
By Lemma \ref{lem_ext}, $u|_X$ can be extended to a continuous function $v$ on $\mybar{X}$. Since $u,v$ are quasi continuous on $\mybar{X}$ and $u=v,m$-a.e., we have $u=v$ q.e., we can take $u$ as $v$. Hence $u$ can be taken continuous, $u\in\Ff_{\mybar{X}}$, $(\Ff_{\mybar{X}})_e\subseteq\Ff_{\mybar{X}}$.
\end{proof}

\begin{myrmk}
It is proved in \cite[Proposition 2.9]{HK06} that a result of above type holds in much more general frameworks.
\end{myrmk}

\begin{proof}[Proof of Theorem \ref{thm_check_E}]
By Equation (\ref{eqn_hit}) and Equation (\ref{eqn_Theta}), we have
$$
\begin{aligned}
\check{\Ee}(u|_{\dd X},u|_{\dd X})&=\Ee_{\mybar{X}}(H_{\dd X}u,H_{\dd X}u)=\Ee_{\mybar{X}}(Hu,Hu)\\
&=\frac{1}{2}\sum_{x,y\in X}c(x,y)(Hu(x)-Hu(y))^2\\
&\asymp\int_K\int_K\frac{(u(x)-u(y))^2}{|x-y|^{\alpha+\beta}}\nu(\md x)\nu(\md y),
\end{aligned}
$$
and
$$
\begin{aligned}
\check{\Ff}&=(\Ff_{\mybar{X}})_e|_{\dd X}\cap L^2(\dd X;\nu)=\Ff_{\mybar{X}}|_{\dd X}\cap L^2(\dd X;\nu)\\
&=\myset{u|_{\dd X}\in L^2(\dd X;\nu):u\in C(\mybar{X}),\sum_{x,y\in X}c(x,y)(u(x)-u(y))^2<+\infty}\\
&=\myset{u|_{\dd X}:u\in C(\mybar{X}),\sum_{x,y\in X}c(x,y)(u(x)-u(y))^2<+\infty}.
\end{aligned}
$$
For all $u|_{\dd X}\in\check{\Ff}$, we have $H_{\dd X}u=Hu\in(\Ff_{\mybar{X}})_e=\Ff_{\mybar{X}}$, $u|_{\dd	X}\in C(\dd X)=C(K)$ and
$$
\begin{aligned}
\int_K\int_K\frac{(u(x)-u(y))^2}{|x-y|^{\alpha+\beta}}\nu(\md x)\nu(\md y)&\asymp\frac{1}{2}\sum_{x,y\in X}c(x,y)(Hu(x)-Hu(y))^2\\
&=\Ee_{\mybar{X}}(Hu,Hu)<+\infty,
\end{aligned}
$$
that is, $\check{\Ff}\subseteq\text{RHS}$. On the other hand, for all $u\in\text{RHS}$, we have $Hu$ satisfies
$$\frac{1}{2}\sum_{x,y\in X}c(x,y)(Hu(x)-Hu(y))^2\asymp\int_K\int_K\frac{(u(x)-u(y))^2}{|x-y|^{\alpha+\beta}}\nu(\md x)\nu(\md y)<+\infty.$$
By Lemma \ref{lem_ext}, $Hu$ can be extended to a continuous function $v$ on $\mybar{X}$, then $v\in C(\mybar{X})$, by Lemma \ref{lem_bdy_har}, we have $v|_{\dd X}=u$.
$$\frac{1}{2}\sum_{x,y\in X}c(x,y)(v(x)-v(y))^2=\frac{1}{2}\sum_{x,y\in X}c(x,y)(Hu(x)-Hu(y))^2<+\infty,$$
hence $v\in\Ff_{\mybar{X}}$, $u\in\check{\Ff}$, $\text{RHS}\subseteq\check{\Ff}$.
\end{proof}

Then we have following corollary.

\begin{mycor}\label{cor_E_K}
Let
$$
\begin{cases}
&\Ee_K(u,u)=\int_K\int_K\frac{(u(x)-u(y))^2}{|x-y|^{\alpha+\beta}}\nu(\md x)\nu(\md y),\\
&\Ff_K=\myset{u\in L^2(K;\nu):\int_K\int_K\frac{(u(x)-u(y))^2}{|x-y|^{\alpha+\beta}}\nu(\md x)\nu(\md y)<+\infty}.
\end{cases}
$$
If $\beta\in(\alpha,\beta^*)$, then $(\Ee_K,\Ff_K)$ is a regular Dirichlet form on $L^2(K;\nu)$.
\end{mycor}
\begin{proof}
Let
$$
\begin{cases}
&\Ee_K(u,u)=\int_K\int_K\frac{(u(x)-u(y))^2}{|x-y|^{\alpha+\beta}}\nu(\md x)\nu(\md y),\\
&\Ff_K=\myset{u\in C(K):\int_K\int_K\frac{(u(x)-u(y))^2}{|x-y|^{\alpha+\beta}}\nu(\md x)\nu(\md y)<+\infty}.
\end{cases}
$$
By Theorem \ref{thm_check_E}, if $\beta\in(\alpha,\beta^*)$, then $(\Ee_K,\Ff_K)$ is a regular Dirichlet form on $L^2(K;\nu)$. We only need to show that
$$\Ff_K=\myset{u\in L^2(K;\nu):\int_K\int_K\frac{(u(x)-u(y))^2}{|x-y|^{\alpha+\beta}}\nu(\md x)\nu(\md y)<+\infty}.$$
Indeed, it is obvious that $\Ff_K\subseteq\text{RHS}$. On the other hand, since $\beta\in(\alpha,\beta^*)$, by \cite[Theorem 4.11 (\rmnum{3})]{GHL03}, $\text{RHS}$ can be embedded into a H\"older space with parameter $(\beta-\alpha)/2$, hence the functions in $\text{RHS}$ can be modified to be continuous, $\text{RHS}\subseteq\Ff_K$.
\end{proof}

\begin{myrmk}
A more general result is proved in \cite{Kum03}.
\end{myrmk}

\section{Triviality of $\Ff_K$ when $\beta\in(\beta^*,+\infty)$}\label{sec_trivial}

In this section, we show that $\Ff_K$ consists of constant functions if $\lambda\in(0,1/5)$ or $\beta\in(\beta^*,+\infty)$. Hence $d_w=\beta^*=\log5/\log2$.

\begin{mythm}\label{thm_constant}
If $\lambda<{1}/{5}$, then for all continuous function $u$ on $\mybar{X}$ with
\begin{equation}\label{eqn_finite}
C=\frac{1}{2}\sum_{x,y\in X}c(x,y)(u(x)-u(y))^2<+\infty,
\end{equation}
we have $u|_{\dd X}$ is constant.
\end{mythm}

\begin{proof}
By Lemma \ref{lem_ext}, Equation (\ref{eqn_finite}) implies that $u|_X$ can be extended continuously to $\mybar{X}$ which is exactly $u$ on $\mybar{X}$. Assume that $u|_{\dd X}$ is not constant. First, we consider
$$\sum_{x,y\in S_n}c(x,y)(u(\Phi_n(x))-u(\Phi_n(y)))^2.$$
By the proof of Theorem \ref{thm_main}, we have
$$\sum_{x,y\in S_{n+1}}c(x,y)(u(\Phi_{n+1}(x))-u(\Phi_{n+1}(y)))^2\ge\frac{1}{5\lambda}\sum_{x,y\in S_n}c(x,y)(u(\Phi_n(x))-u(\Phi_n(y)))^2.$$
Since $u|_{\dd X}$ is continuous on $\dd X$ and $u|_{\dd X}$ is not constant, there exists $N\ge1$ such that
$$\sum_{x,y\in S_N}c(x,y)(u(\Phi_N(x))-u(\Phi_N(y)))^2>0.$$
Since $\lambda<1/5$, for all $n\ge N$, we have
$$
\begin{aligned}
&\sum_{x,y\in S_n}c(x,y)(u(\Phi_n(x))-u(\Phi_n(y)))^2\\
&\ge\frac{1}{(5\lambda)^{n-N}}\sum_{x,y\in S_N}c(x,y)(u(\Phi_N(x))-u(\Phi_N(y)))^2\to+\infty,
\end{aligned}
$$
as $n\to+\infty$. Next, we consider the relation between
$$\sum_{x,y\in S_n}c(x,y)(u(x)-u(y))^2\text{ and }\sum_{x,y\in S_n}c(x,y)(u(\Phi_n(x))-u(\Phi_n(y)))^2.$$
Indeed
$$
\begin{aligned}
&\sum_{x,y\in S_n}c(x,y)(u(x)-u(y))^2\\
&\le\sum_{x,y\in S_n}c(x,y)\left(\lvert u(x)-u(\Phi_n(x))\rvert+\lvert u(\Phi_n(x))-u(\Phi_n(y))\rvert+\lvert u(\Phi_n(y))-u(y)\rvert\right)^2\\
&\le3\sum_{x,y\in S_n}c(x,y)\left((u(x)-u(\Phi_n(x)))^2+(u(\Phi_n(x))-u(\Phi_n(y)))^2+(u(\Phi_n(y))-u(y))^2\right)\\
&=3\sum_{x,y\in S_n}c(x,y)(u(\Phi_n(x))-u(\Phi_n(y)))^2\\
&+3\sum_{x,y\in S_n}c(x,y)\left((u(x)-u(\Phi_n(x)))^2+(u(\Phi_n(y))-u(y))^2\right).
\end{aligned}
$$
For all $x\in S_n$, there are at most 3 elements $y\in S_n$ such that $c(x,y)>0$ and for all $x,y\in S_n$, $c(x,y)\le\max{\myset{C_1,C_2}}/(3\lambda)^n$. By symmetry, we have
$$
\begin{aligned}
&\sum_{x,y\in S_n}c(x,y)\left((u(x)-u(\Phi_n(x)))^2+(u(\Phi_n(y))-u(y))^2\right)\\
&=2\sum_{x,y\in S_n}c(x,y)(u(x)-u(\Phi_n(x)))^2\\
&\le6\sum_{x\in S_n}\frac{\max{\myset{C_1,C_2}}}{(3\lambda)^n}(u(x)-u(\Phi_n(x)))^2.
\end{aligned}
$$
For all $x\in S_n$, there exists a geodesic ray $[x_0,x_1,\ldots]$ with $|x_k|=k$, $x_k\sim x_{k+1}$ for all $k\ge0$ and $x_n=x$, $x_k\to\Phi_n(x)$ as $k\to+\infty$. For distinct $x,y\in S_n$, the corresponding geodesic rays $[x_0,x_1,\ldots]$, $[y_0,y_1,\ldots]$ satisfy $x_k\ne y_k$ for all $k\ge n$. Then
$$
\begin{aligned}
\frac{1}{(3\lambda)^n}(u(x)-u(\Phi_n(x)))^2&\le\frac{1}{(3\lambda)^n}\left(\sum_{k=n}^\infty|u(x_k)-u(x_{k+1})|\right)^2\\
&=\left(\sum_{k=n}^\infty\frac{1}{(3\lambda)^{n/2}}|u(x_k)-u(x_{k+1})|\right)^2\\
&=\left(\sum_{k=n}^\infty{(3\lambda)^{(k-n)/2}}\frac{1}{(3\lambda)^{k/2}}|u(x_k)-u(x_{k+1})|\right)^2\\
&\le\sum_{k=n}^\infty(3\lambda)^{k-n}\sum_{k=n}^\infty\frac{1}{(3\lambda)^k}(u(x_k)-u(x_{k+1}))^2\\
&=\frac{1}{1-3\lambda}\sum_{k=n}^\infty c(x_k,x_{k+1})(u(x_k)-u(x_{k+1}))^2,
\end{aligned}
$$
hence
$$
\begin{aligned}
\sum_{x\in S_n}\frac{1}{(3\lambda)^n}(u(x)-u(\Phi_n(x)))^2&\le\frac{1}{1-3\lambda}\sum_{x\in S_n}\sum_{k=n}^\infty c(x_k,x_{k+1})(u(x_k)-u(x_{k+1}))^2\\
&\le\frac{1}{1-3\lambda}\left(\frac{1}{2}\sum_{x,y\in X}c(x,y)(u(x)-u(y))^2\right)\\
&=\frac{1}{1-3\lambda}C.
\end{aligned}
$$
We have
$$
\sum_{x,y\in S_n}c(x,y)\left((u(x)-u(\Phi_n(x)))^2+(u(\Phi_n(y))-u(y))^2\right)\le\frac{6\max{\myset{C_1,C_2}}}{1-3\lambda}C,
$$
and
$$
\sum_{x,y\in S_n}c(x,y)(u(x)-u(y))^2\le 3\sum_{x,y\in S_n}c(x,y)(u(\Phi_n(x))-u(\Phi_n(y)))^2+\frac{18\max{\myset{C_1,C_2}}}{1-3\lambda}C.
$$
Similarly, we have
$$
\sum_{x,y\in S_n}c(x,y)(u(\Phi_n(x))-u(\Phi_n(y)))^2\le 3\sum_{x,y\in S_n}c(x,y)(u(x)-u(y))^2+\frac{18\max{\myset{C_1,C_2}}}{1-3\lambda}C.
$$
Since
$$\lim_{n\to+\infty}\sum_{x,y\in S_n}c(x,y)(u(\Phi_n(x))-u(\Phi_n(y)))^2=+\infty,$$
we have
$$\lim_{n\to+\infty}\sum_{x,y\in S_n}c(x,y)(u(x)-u(y))^2=+\infty.$$
Therefore
$$C=\frac{1}{2}\sum_{x,y\in X}c(x,y)(u(x)-u(y))^2=+\infty,$$
contradiction! Hence $u|_{\dd X}$ is constant.
\end{proof}

\begin{mythm}\label{thm_ub}
If $\lambda\in(0,1/5)$ or $\beta\in(\beta^*,+\infty)$, then $(\Ee_K,\Ff_K)$ on $L^2(K;\nu)$ is trivial, that is, $\Ff_K$ consists of constant functions. Hence walk dimension of SG $d_w=\beta^*=\log5/\log2$.
\end{mythm}

\begin{proof}
For all $u\in\Ff_K$, let $v=Hu$ on $X$, then we have
$$\frac{1}{2}\sum_{x,y\in X}c(x,y)(v(x)-v(y))^2<+\infty.$$
Since $\lambda<1/5<1/3$, by Lemma \ref{lem_ext}, $v$ on $X$ can be extended continuously to $\mybar{X}$ still denoted by $v$. By Lemma \ref{lem_bdy_har}, we have $v|_{\dd X}=u$. By Theorem \ref{thm_constant}, we have $v|_{\dd X}$ is constant, hence $u$ is constant. $\Ff_K$ consists of constant functions.
\end{proof}


\def\cprime{$'$}

\end{document}